\documentclass[dvipsnames,reqno,10pt]{amsart}

\usepackage{amsmath, amssymb ,amsthm, amsfonts, amsgen, mathtools}
\usepackage{bbm}
\usepackage{bm}
\usepackage{tikz}
\usepackage{subcaption}
\usetikzlibrary{intersections,calc,patterns}
\usepackage{enumitem}
\usepackage[colorlinks=true, linkcolor=blue, citecolor=blue]{hyperref}
\usepackage[normalem]{ulem}
\usepackage{cancel}
\usepackage[noadjust]{cite}

\numberwithin{equation}{section}

\DeclareMathOperator{\Div}{div}
\newcommand{\loc}{{\rm loc}}

\newcommand{\weakly}{\rightharpoonup}

\newcommand{\Ecal}{{\mathcal{E}}}
\newcommand{\Ical}{{\mathcal{I}}}
\newcommand{\Scal}{{\mathcal{S}}}
\newcommand{\Hcal}{{\mathcal{H}}}
\newcommand{\Acal}{{\mathcal{A}}}
\newcommand{\Qcal}{{\mathcal{Q}}}
\newcommand{\Pcal}{{\mathcal{P}}}
\newcommand{\eps}{ \varepsilon}

\newcommand{\N}{{\mathbb{N}}}
\newcommand{\R}{{\mathbb{R}}}
\newcommand{\ffi}{\varphi}
\newcommand{\one}{\mathbbm{1}}
\newcommand{\qc}{{\rm qc}}

\newcommand{\dd}{{\,d}}

\newcommand\restrict[1]{\raisebox{-.5ex}{$|$}_{#1}}
\renewcommand{\epsilon}{\varepsilon}
\newcommand{\norm}[1]{\|#1\|}

\newcommand{\abs}[1]{|#1|}
\newcommand{\abslr}[1]{\left| #1 \right|}
\newcommand{\qand}{\quad\text{and}\quad}
\newcommand{\Wmin}{\overline{W}}
\renewcommand{\phi}{\varphi}
\renewcommand{\d}{\delta}
\newcommand{\e}{\varepsilon}
\renewcommand{\O}{\Omega}
\DeclareMathOperator{\supp}{supp}
\newcommand{\He}{H^{\rho,p}_{\delta^\epsilon}(\Omega;\R^3)}
\newcommand{\Hd}[1]{H^{\rho,p}_{\delta}(#1;\R^3)}
\newcommand{\Hdk}[1]{H^{\rho,p}_{\delta_k}(#1;\R^3)}
\newcommand{\Hzero}{H^{\rho,p}_{\delta^0}(\Omega;\R^3)}

\newcommand{\Ne}{N^{\rho,p}_{\delta^\epsilon}(\Omega;\R^3)}
\newcommand{\Nd}[1]{N^{\rho,p}_{\delta}(#1;\R^3)}
\newcommand{\Ndk}[1]{N^{\rho,p}_{\delta_k}(#1;\R^3)}
\newcommand{\Nzero}{N^{\rho,p}_{\delta^0}(\Omega;\R^3)}

\newcommand{\Qbar}{\bar{Q}_{\rho,\bar{\d}}}

\newcommand{\Qcale}{\Qcal_{\rho,\d^\e}}
\newcommand{\Qcald}{\Qcal_{\rho,\d}}
\newcommand{\Qcaldk}{\Qcal_{\rho,\d_k}}
\newcommand{\Qcalzero}{\Qcal_{\rho,\d^0}}
\newcommand{\Pcale}{\Pcal_{\rho,\d^\e}}
\newcommand{\Pcald}{\Pcal_{\rho,\d}}
\newcommand{\Pcaldk}{\Pcal_{\rho,\d_k}}

\newcommand{\Pcalzero}{\Pcal_{\rho,\d^0}}
\newcommand{\De}{D_{\rho,\d^\e}}
\newcommand{\Dzero}{D_{\rho,\d^0}}
\newcommand{\Dbar}{\bar{D}_{\rho,\d^0}}
\newcommand{\Dd}{D_{\rho,\d}}
\newcommand{\Ddk}{D_{\rho,\d_k}}
\newcommand{\Dtwo}{D_{\bar{\rho},\bar{\d}^0}}
\newcommand{\Dtwod}{D_{\bar{\rho},\bar{\d}}}
\newcommand{\Dtwoe}{D_{\bar{\rho},\bar{\d}^\e}}
\newcommand{\dive}{\Div_{\rho,\d^\e}}
\newcommand{\divzero}{\Div_{\rho,\d^0}}
\newcommand{\divd}{\Div_{\rho,\d}}

\newcommand{\Pcaltwo}{\Pcal_{\bar{\rho},\bar{\d}^0}}

\newcommand{\Pcaltwoe}{\Pcal_{\bar{\rho},\bar{\d}^\e}}
\newcommand{\Qcaltwoe}{\Qcal_{\bar{\rho},\bar{\d}^\e}}
\newcommand{\Pcaltwod}{\Pcal_{\bar{\rho},\bar{\d}}}
\newcommand{\Qcaltwod}{\Qcal_{\bar{\rho},\bar{\d}}}
\DeclareMathOperator*{\argmin}{arg\,min}
\DeclareMathOperator{\dist}{dist}

\newcommand{\black}{\color{black}}

\newtheoremstyle{thmlemcorr}{10pt}{10pt}{\itshape}{}{\bfseries}{.}{10pt}{{\thmname{#1}\thmnumber{ #2}\thmnote{ (#3)}}}
\newtheoremstyle{thmlemcorr*}{10pt}{10pt}{\itshape}{}{\bfseries}{.}\newline{{\thmname{#1}\thmnumber{ #2}\thmnote{ (#3)}}}
\newtheoremstyle{defi}{10pt}{10pt}{\itshape}{}{\bfseries}{.}{10pt}{{\thmname{#1}\thmnumber{ #2}\thmnote{ (#3)}}}
\newtheoremstyle{remexample}{10pt}{10pt}{}{}{\bfseries}{.}{10pt}{{\thmname{#1}\thmnumber{ #2}\thmnote{ (#3)}}}
\newtheoremstyle{ass}{10pt}{10pt}{}{}{\bfseries}{.}{10pt}{{\thmname{#1}\thmnumber{ A#2}\thmnote{ (#3)}}}

\theoremstyle{thmlemcorr}
\newtheorem{theorem}{Theorem}
\numberwithin{theorem}{section}
\newtheorem{lemma}[theorem]{Lemma}
\newtheorem{corollary}[theorem]{Corollary}
\newtheorem{proposition}[theorem]{Proposition}

\theoremstyle{thmlemcorr*}
\newtheorem{theorem*}{Theorem}
\newtheorem{lemma*}[theorem]{Lemma}
\newtheorem{corollary*}[theorem]{Corollary}
\newtheorem{proposition*}[theorem]{Proposition}
\newtheorem{problem*}[theorem]{Problem}
\newtheorem{conjecture*}[theorem]{Conjecture}

\theoremstyle{defi}
\newtheorem{definition}[theorem]{Definition}

\theoremstyle{remexample}

\theoremstyle{remexample}
\newtheorem{remarkx}[theorem]{Remark}
\newenvironment{remark}
  {\pushQED{\qed}\remarkx}
  {\popQED\endremarkx}
\newtheorem{examplex}[theorem]{Example}
\newenvironment{example}
  {\pushQED{\qed}\examplex}
  {\popQED\endexamplex}

\title[Derivation of membrane models in nonlocal hyperelasticity]{Derivation of variational membrane models in the context of anisotropic nonlocal hyperelasticity}

\author{Dominik Engl}
\address{Mathematisch-Geographische Fakult\"at, Katholische Universit\"at Eichst\"att-Ingolstadt, Ostenstra{\ss}e 28, 85071 Eichst\"att, Germany}
\email{dominik.engl@ku.de}

\author{Anastasia Molchanova}
\address{Institute of Analysis and Scientific Computing, TU Wien,
Wiedner Hauptstra\ss e 8--10, 1040 Wien, Austria}
\email{anastasia.molchanova@tuwien.ac.at}

\author{Hidde Sch\"onberger}
\address{Research Institute in Mathematics and Physics, Université catholique de Louvain, Chemin du Cyclotron 2, 1348 Louvain-la-Neuve, Belgium}
\email{hidde.schonberger@uclouvain.be}

\begin{document}

\begin{abstract}    
Motivated by the analysis of thin structures, we study the variational dimension reduction of hyperelastic energies involving nonlocal gradients to an effective membrane model. When rescaling the thin domain, isotropic interaction ranges naturally become anisotropic, leading to the development of a theory for anisotropic nonlocal gradients with direction-dependent interaction ranges. Unlike existing nonlocal derivatives with finite horizon, which are defined via interaction kernels supported on balls of positive radius, our formulation is based on ellipsoidal interaction regions whose principal radii may vanish independently. This yields a unified framework that interpolates between fully nonlocal, partially nonlocal, and purely local models. Employing these tools, we present a $\Gamma$-convergence analysis for the nonlocal thin-film energies. The limit functional retains the structural form of the classical membrane energy, and the classical local model is recovered precisely when all interaction radii vanish.
    \medskip
    
    \noindent\textsc{MSC (2020):} 49J45 (primary), 74G65, 74K15, 35R11, 74A70, 70G75
    \medskip

    \noindent\textsc{Keywords:} nonlocal gradients, nonlocal hyperelasticity, localization, dimension reduction, membrane, $\Gamma$-convergence
    \medskip
 
    \noindent\textsc{Date:} \today.
\end{abstract}

\maketitle

\section{Introduction}

Peridynamics is a nonlocal extension of classical continuum mechanics introduced in~\cite{Sil2000}, that naturally admits discontinuities such as fractures and cracks. 
Unlike classical elasticity, which relies on derivatives of deformations or displacements and hence becomes ill-posed near singularities, peridynamics formulates the governing equations of continuum mechanics in a nonlocal form. 
Thus, the energy is based on pairwise force interactions between material points within a defined interaction range known as the peridynamic horizon.

Peridynamics is typically classified into two categories. \emph{Bond-based peridynamics}, introduced in~\cite{Sil2000}, restricts force interactions to pairwise bonds between individual material points, where the force magnitude depends only on the relative displacement between two points. \textit{State-based peridynamics}, proposed in~\cite{SilEptWecXuAsk2007}, generalizes this framework by allowing interaction forces to depend on the deformation state of all bonds connected to the endpoints, not merely the bond stretch itself.
The second approach allows for more realistic modeling of material behavior.
However, the mathematical structure of state-based models is more complex, particularly in nonlinear and variational settings.
For an overview of recent developments in theory, as well as their numerical implementations, the reader is referred to the reviews \cite{DimCocAleFanPol2022,DorRenZhuSilMadRab2024,HanZhaWanLuJia2019,LadGon2021, OteOte2024} and references therein.

\subsection*{Nonlocal hyperelasticity} 
This work addresses a specific class of state-based peridynamics called \textit{nonlocal hyperelasticity}. In this framework, the material response is governed by an energy functional that mirrors the classical hyperelastic setting but replaces the local gradient $\nabla$ by a nonlocal counterpart $D_\rho$ defined by
\[
D_\rho u (x) :=\int_{\R^n} \frac{u(y)-u(x)}{\abs{y-x}}\otimes\frac{y-x}{\abs{y-x}}\rho(y-x)\,dy \quad \text{for $x \in \R^n$},
\]
with a suitable radial interaction kernel $\rho\colon\R^n \setminus \{0\} \to [0,\infty)$ and $u\colon\R^n \to \R^n$. Precisely, we consider stored energies of the form
\begin{align}\label{nonlocal_energy}
    u\mapsto \int_{\Omega}W(D_\rho u)\dd x
\end{align}
where $u$ describes the deformation of the reference configuration $\Omega\subset \R^n$, and $W\colon \R^{n\times n} \to [0,\infty)$ denotes the material energy density. These models were already proposed in \cite[Section~18]{SilEptWecXuAsk2007}, but have only recently been subjected to a rigorous variational analysis. The first work to investigate the Sobolev spaces associated with the nonlocal gradients was \cite{ShS15} (see also~\cite{ShS18,CoS19}), in which the Riesz fractional gradient was considered that arises by choosing the interaction kernel $\rho = \abs{\cdot}^{-(n+s-1)}$ for some fractional exponent $s \in (0,1)$. There, the authors established fractional Poincar\'e inequalities and compact embeddings, and addressed existence of minimizers for convex energy densities. These existence results were later generalized to polyconvex and quasiconvex energy densities \cite{BCM20,KrS22}, which are typically employed in applications of hyperelasticity. To better align with peridynamic models, the extension to finite-horizon fractional gradients was considered in \cite{BCM23,CKS23}. More recently, a theory for general compactly supported interaction kernels $\rho$ was developed \cite{BMS24}, a setting that we also adopt in this paper. Beyond the results on the function spaces and the existence of minimizers, it has furthermore been shown in \cite{CKS25} (see also~\cite{MeS15, BGS26}) that nonlocal models $\Gamma$-converge to local ones as the horizon vanishes, in particular for families of kernels $\rho_\d:= \delta^{-n}\rho(\cdot/\d)$ for $\d>0$.

\subsection*{Classical dimension reduction} We are particularly interested in the study of thin structures, which are known to show fundamentally different responses to external loading compared to bulk materials. 
A natural way to analytically capture these effects within classical nonlinear hyperelasticity is to study the asymptotic behavior of three-dimensional energy functionals on domains whose thickness tends to zero and to identify the resulting lower-dimensional models. 
Depending on the underlying scaling, one obtains qualitatively different theories in the limit. 
Early rigorous results in this direction include limits that yield string \cite{ABP91} and membrane models \cite{LeR95}.
The subsequent analysis of thin structures was strongly shaped by the seminal paper \cite{FJM02}, which led, among other things, to a systematic derivation of plate theories in several scaling regimes \cite{FJM02,FJM06} and the generalization to non-Euclidean elasticity \cite{Lew23}, as well as the development of rod theories \cite{MoM03,MoM04,Sca06,Sca09}.
In the last two decades, variational dimension reduction has seen extensions in several different directions.
A recurring theme is the incorporation of non-convex differential constraints, motivated for instance by incompressibility \cite{CoD06,CoD09,EnK21,EnK21b}, non-self-interpenetration of matter and (global) invertibility \cite{LPS15,OlR17,Sch07,BCHK25}, or $\Acal$-free fields \cite{KrR15,Kre17} with applications in elasticity or micromagnetism.
As the latter indicates, thin-structure limits have also been pursued well beyond the purely elastic setting. 
Further extensions include elastoplasticity \cite{EKK24,Dav13,Dav14}, magnetoelasticity \cite{Bre21,BrK23}, liquid crystal elastomers \cite{CPB15,BPP25}, or fracture mechanics \cite{Bab06b,Sch17,ARS23}, to name just a few.

Much of the existing literature on lower-dimensional peridynamic models is based on heuristic reductions or formal asymptotics, often complemented by numerical simulations: peridynamic membranes were introduced in \cite{SilBob2005}, plate models were studied in \cite{TaySte2015,OGraFos2014,YanVazDiyOteOte2020, CavCutDayDesFra2023}, while shell theory is considered in \cite{BehAlaTraBaz2022,ChoRoyRoyRed2016}.
In contrast, a rigorous variational thin-structure theory for peridynamics is still scarce and, at present, is represented by only a small number of recent contributions, notably \cite{BrS25,BPS25}. In \cite{BrS25}, the authors identify the critical scaling for the $W^{s,2}$-Gagliardo seminorm for $s>\frac12$ that yields nontrivial lower-dimensional limit models, namely $(1-s)\e^{2s-3}$, thereby making explicit the coupling between the thin geometry and the fractional nature of the interaction. Moreover, they establish a Bourgain–Brezis–Mironescu-type \cite{BBM01} result in the joint limit $(\e,s)\to (0,1)$.
The follow-up paper \cite{BPS25} treats smaller fractional parameters $s<\frac12$. Among other results, it shows that scaling the energy by $\e^{-2}$ increases the fractional parameter of the lower-dimensional limit model.
Despite these recent advances in the fractional, infinite-horizon setting, a $\Gamma$-convergence result for finite horizon appears to be unexplored.

\subsection*{Contributions of the current work} The contribution of this work is a thin-film analysis of finite-horizon nonlocal hyperelastic materials, in which
the stored energy is given by \eqref{nonlocal_energy} in the reference configuration $\Omega^\e:=\omega\times (0,\e)$, where the cross-section $\omega\subset \R^2$ is a bounded Lipschitz domain and $\e>0$ represents the width. In parallel, we are currently working on a corresponding analysis for bond-based peridynamics similar to \cite{MenSco2024,BelMorCorPed2015}. 
As customary, the $\e$-dependence of the domain is removed by a linear rescaling in the out-of-plane direction. 
Under this rescaling, a radially symmetric interaction neighborhood for the nonlocal gradient becomes anisotropic (ellipsoidal), and the resulting non-radial structure is not covered by the existing theory tailored to isotropic horizons.
Accordingly, Section~\ref{sec:nonlocal_gradients} develops a general theory for anisotropic nonlocal derivatives suited to thin geometries. For simplicity, we replace the standard omnidirectional horizon $\delta\in (0,1]$ with a pair of (possibly vanishing) horizons $\d=(\bar{\d},\d_3)\in[0,1]^2$, where $\bar{\d}$ controls the in-plane interaction radius and $\d_3$ the out-of-plane interaction radius, see Figure \ref{fig:ellipsoid}. This yields a unified framework that interpolates between fully nonlocal, partially nonlocal, and purely local regimes; extensions to higher dimensions and multiple radii are discussed in Remark~\ref{rem:higher_dimensions}. 
\begin{figure}[h!]
    \centering
    \begin{tikzpicture}[scale=1.7]	
        \draw (-1.7,0) -- (-1.2,0);
        \draw[densely dotted] (-1.2,0) -- (1.2,0);
        \draw[->] (1.2,0) -- (1.7,0) node [right] {\small $x_1$};
        \draw (-149:-1.5) -- (-149:-1);
        \draw[densely dotted] (-149:-1) -- (0,0);
        \draw[->] (0,0) -- (-149:1.5) node[below left] {\small $x_2$};
        \draw (0,-1.2) -- (0,-.75);
        \draw [densely dotted](0,-.75) -- (0,.75);
        \draw[->] (0,.75) -- (0,1.2) node [right] {\small $x_3$};
        	
        \draw[thick] (-1.2,0) arc (180:360:1.2cm and 0.3cm);
        \draw[thick,dashed] (-1.2,0) arc (180:0:1.2cm and 0.3cm);
        \draw[thick] (0,.75) arc (90:270:0.5cm and .75cm);
        \draw[thick,dashed] (0,.75) arc (90:-90:0.5cm and .75cm);
        \draw[thick] (0,0) ellipse (1.2cm and .75cm);
        \shade[ball color=blue!10!white,opacity=0.20] (0,0) ellipse (1.2cm and .75cm);
        \draw[red,thick,->] (0,0) -- (0,.25) node[above left] {\small $\delta_3$} -- (0,.75) ;
        \draw[ForestGreen,thick,->] (0,0) -- (1.2,0) node[below right] {\small $\bar{\delta}$};
    	\draw[ForestGreen,thick,->] (0,0) -- (-149:.55) node[above left] {\small $\bar{\delta}$};
    \end{tikzpicture}
    \caption{An illustration of an ellipsoid with principal radii $\bar{\d},\bar{\d},\d_3>0$.}
    \label{fig:ellipsoid}
\end{figure}
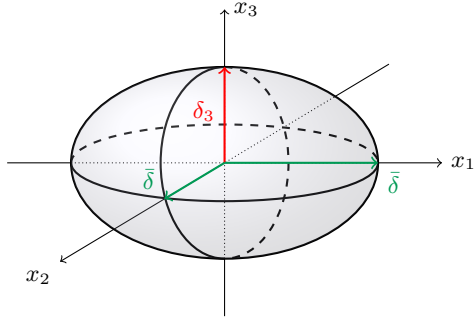

The basic space of interest for our analysis is
\begin{align}\label{Hrho_intro}
	\Hd{O}:=\{u\in L^p(\R^3;\R^3) : u = 0 \text{ on } (O_{\d})^c \text{ and } \Dd u \in L^p(O;\R^{3\times 3})\},
\end{align}
where $\Dd u$ is a nonlocal gradient with ellipsoidal interaction neighborhood, see~\eqref{Dzero}, $O \subset \R^3$ is an open set, and 
$$O_{\d} := O +  T_{\d}B_1(0)\quad \text{with}\quad T_{\d}:=\begin{pmatrix}\bar{\d} & 0 & 0\\ 0 & \bar{\d} & 0\\ 0 & 0 & \d_3\end{pmatrix},$$
with $B_1(0)$ being the unit ball.
We impose the assumption $u=0$ on $(O_{\d})^c$ for convenience since these values do not affect $\Dd u$ on $O$ and it enables the embedding $\Hd{O}\subset L^p(\R^3;\R^3)$ independently of $\d$.

The core result in Section \ref{sec:nonlocal_gradients} is the translation mechanism between $H^{\rho,p}_\d$ and $W^{1,p}$ for any $\d\in[0,1]^2$ in Proposition \ref{prop:translation}, which was first developed in \cite{CKS23} for isotropic positive horizons. It constitutes an isomorphism that turns the local gradient into the nonlocal one, and provides a powerful tool for reducing proofs in the nonlocal setting to the well-known local case. While the definition of the translation operator arises from the isotropic one via a straightforward change of variables, a key challenge is to obtain estimates uniformly in $\d$, as the usual Fourier methods do not carry over to the anisotropic case. Additionally, new difficulties arise in the degenerate case where one component of $\d$ vanishes.
There, we rely on a careful approximation by strictly positive horizons $(\d_k)_k\subset (0,1]^2$ and on the density of $C_c^\infty$-functions in $H^{\rho,p}_\d$ (see Propositions~\ref{prop:density} and \ref{prop:density_full_space}), which we prove before via a suitable Leibniz rule (Proposition \ref{lem:product}). Compared to the existing literature, establishing these two results becomes substantially more delicate in the partially local setting. The translation mechanism provides several immediate consequences needed for the membrane analysis, most notably a Poincar\'e-type inequality and a reduction principle for the out-of-plane variable.
More precisely, since there exist functions with vanishing nonlocal gradient on $O$ that are not constant (cf.~Remark~\ref{rem:N} and \cite{KrS24}), it is not immediate that $(\Dd u)e_3=0$ enforces independence of $x_3$. 
We show, however, that there exists $h\in \Hd{O}$ with $\Dd h = 0$ such that $u-h$ is independent of $x_3$, see Lemma \ref{le:constant_x3}. 
In fact, Lemma \ref{le:2dgradient} then guarantees the fact that the first two columns of $\Dd(u-h)$ actually coincide with a (non)local two-dimensional in-plane gradient.

In the thin-film limit of Section~\ref{sec:thin_films}, the out-of-plane rescaling maps the physical horizon $\d\in [0,1]^2$ to the rescaled horizon
$$\d^\e:=\bigg(\bar{\d},\frac{\d_3}{\e}\bigg)\in [0,1]^2,$$
where neither $\d$ nor $\d^\e$ is assumed to be isotropic. For technical reasons, we also implicitly assume that $\d_3\leq \e$. 
While $\d_3 \sim \e$ can be treated similarly, the regime $\d_3\gg \e$ is beyond the scope of this work. The main obstruction is geometric: these nonlocal energies require deformations to be defined on the (in practice only slightly) enlarged interaction domain $\Omega_{\d^\e}$, where $\Omega=\omega \times (0,1)$. However, this set explodes in the out-of-plane direction if $\frac{\d_3}{\e} \gg 1$. A particular consequence is that $\delta_3=\delta_3(\epsilon)$ should depend on $\epsilon$. In fact, we also allow the cross-section horizon $\bar{\d}=\bar{\d}(\epsilon)$ to depend on $\e$, which makes it possible to recover both nonlocal and local dimensionally reduced models, see \eqref{eq:cases}. We mention that the tools developed in Section~\ref{sec:nonlocal_gradients} make it possible to treat anisotropic varying horizon limits without dimension reduction as well, but we do not carry this out here (see~\cite{CKS25,BGS26} for the isotropic case).

The rescaled stored elastic energy, normalized per unit volume, is now given by
\begin{align}\label{energies_intro}
   \Ical_\eps\colon L^p(\R^3;\R^3)\to [0,\infty],\ u\mapsto 
    \begin{cases}
        \displaystyle\int_{\Omega} W\big((D_{\rho,\d^\e} u) T_{(1,\e)}^{-1}\big)\dd x &\text{ if } u\in \He ,\\
        \infty  & \text{ otherwise}.
    \end{cases}
\end{align} 
We emphasize that the choice $\d=0$ recovers the classical hyperelastic setting \cite{LeR95} up to boundary conditions, which we omit here.
With the theoretical framework of anisotropic nonlocal gradients in place, the dimension reduction becomes conceptually straightforward. Given $\d^\e:=(\bar{\d},\frac{\d_3}{\e})\to \d^0\in[0,1]^2$, we prove in Theorem \ref{th:state} that $(\Ical_\e)_\e$ from \eqref{energies_intro} $\Gamma$-converges with respect to the strong topology in $L^p(\R^3;\R^3)$ to
\begin{align}\label{limit_intro}
	\Ical\colon L^p(\R^3;\R^3)\to [0,\infty],\ u\mapsto 
    \begin{cases}
		\displaystyle \int_\Omega \Wmin^\qc(\Dbar u)\dd x & \text{ if }u\in \Hzero, (\Dzero u)\restrict{\Omega}e_3 =0,\\
    	\infty  &\text{ otherwise.}
    \end{cases}
\end{align}
Here, $\Dbar$ denotes the first two columns of $\Dzero$, the density $\Wmin\colon \R^{3\times 2} \to [0,\infty)$ emerges from $W$ through minimization in the last column, and $(\cdot)^\qc$ stands for the quasiconvex envelope. The lower-dimensional character of \eqref{limit_intro} becomes particularly transparent once one chooses a representative $\bar{u}$ of $u$ that is independent of $x_3$ as in Remark \ref{rem:2drep}, so that we can express
\begin{align}\label{eq:niceGammalim}
        \Ical(u) = \int_{\omega} \Wmin^\qc( \Dtwo \bar{u}) \dd \bar{x},
\end{align}
where $\Dtwo$ is the in-plane nonlocal gradient of horizon $\bar{\d}^0$ with respect to the reduced kernel $\bar{\rho}$ from Lemma~\ref{le:2dgradient}. It is worth pointing out that this in-plane gradient only depends on the limit $\bar{\delta}^0$ and is thus independent of the relation between $\d$ and $\e$ as long as $\d_3\leq \e$.
The $\Gamma$-convergence is complemented by a compactness result, which holds up to subtraction of some functions with vanishing (non)local gradient; if $\d=0$, this reduces to subtracting mean values. 
The proofs rely heavily on the translation mechanism, which allows us to pass back and forth between the nonlocal and local settings.

To illustrate the result, we highlight the two main cases of interest:
\begin{equation}\label{eq:cases}
\text{(i)} \ \ \delta(\epsilon) = (1,\epsilon) \quad \text{and} \quad \text{(ii)} \ \ \d(\epsilon) = (\epsilon,\epsilon).
\end{equation}
The former reflects a physical horizon that is anisotropically scaled in the vertical direction equal to the thickness of the domain; one can think of this as arising from vertically squeezing an elastic material with isotropic horizon. The rescaled horizon $\d^\e=(1,1)$ has no more dependence on $\epsilon$, and one obtains from \eqref{eq:niceGammalim} a nonlocal $\Gamma$-limit given by
\begin{align*}
        \Ical(u) = \int_{\omega} \Wmin^\qc( D_{\bar{\rho}} \bar{u}) \dd \bar{x}.
\end{align*}
The case (ii) corresponds to an isotropic physical horizon which scales like $\epsilon$ in all directions; this can be interpreted as zooming out from an infinite length strip with finite thickness and a fixed horizon. We find that $\bar{\delta}^0 = 0$, meaning that \eqref{eq:niceGammalim} becomes the local hyperelastic functional
\begin{align*}
        \Ical(u) = \int_{\omega} \Wmin^\qc( \nabla \bar{u}) \dd \bar{x}.
\end{align*}

A technical and modeling nuisance for the energies \eqref{energies_intro} and \eqref{limit_intro} is their invariance under the addition of functions with vanishing gradient. 
In the setting $\d=0$, this invariance is harmless, since the only functions with zero gradient are constants.
The nonlocal setting, however, gives rise to highly nontrivial functions $h\in \He$ with $D_{\rho,\d^\e} h=0$, see~Remark~\ref{rem:N} and \cite{KrS24}; as a consequence, such large deformations are energetically invisible.
In the spirit of \cite[Section 6]{Sil17}, we therefore add a stabilization term that penalizes such contributions and, in particular, discourages them from appearing along minimizing, or almost-minimizing, sequences:
\[
\Scal_\epsilon (u):= \int_{\Omega_{\d^\e}}\abs{u}^p\,dx + \int_{\Omega_{\d^\e}}\int_{\Omega_{\d^\e}}\frac{\abs{u(\bar{x},x_3)-u(\bar{x},x_3')}^p}{\abs{\epsilon(x_3-x_3')}}\,dx\,dx' \quad \text{for $u \in L^p(\R^3;\R^3)$}.
\]
The first term penalizes large deformations with zero nonlocal gradient, while the second, purely out-of-plane, nonlocal term enforces convergence to functions that are constant in the third variable (and hence vanishes on the limiting class). We conclude this paper with a final $\Gamma$-convergence result, namely for $(\Ical_\e + \Scal_\e)_\e$, which is complemented with a weak compactness result that does not require the subtraction of functions with zero (nonlocal) gradient.

This paper is organized as follows: In Section \ref{sec:nonlocal_gradients}, we establish the full theory for anisotropic nonlocal gradients, while Section \ref{sec:thin_films} deals with the asymptotic variational analysis of thin films in nonlocal hyperelasticity.\black

\subsection{Notation}
In this notation section we select $n\in\{2,3\}$. The standard unit vectors in $\R^n$ are denoted by $e_1,\ldots, e_n$ and we write $x\cdot y$ for the scalar product of two vectors $x,y\in\R^n$ and set $|x|:=\sqrt{x\cdot x}$.
Since this article is motivated by the analysis of three-dimensional thin films, we encounter a natural anisotropy of the reference configuration in the sense that it is much smaller in one direction. This often leads us to study in-plane and out-of-plane quantities separately.
We thus abbreviate the first two components of a vector $x=(x_1,x_2,x_3)\in \R^3$ by $\bar{x}:=(x_1,x_2)\in\R^2$.
The space $\R^{m\times n}$ of $m\times n$-matrices with $m\in\{2,3\}$ shall be equipped with the standard Frobenius norm given by $|A|:=\sqrt{\mathrm{Tr}(A^TA)}$ for $A\in\R^{m\times n}$, where $A^T$ is the transpose of $A\in \R^{m\times n}$ and $\mathrm{Tr}$ is the trace. We frequently employ the notation
$$T_r:=\begin{pmatrix}\bar{r} & 0 & 0\\ 0 & \bar{r} & 0\\ 0 & 0 & r_3\end{pmatrix} \quad \text{for}\quad r=(\bar{r},r_3)\in [0,\infty)^2,$$
where $\bar{r}\geq 0$ and $r_3\geq 0$ represent (possibly vanishing) in- and out-of-plane distances.

Given two sets $U,V\subset \R^n$, their usual Minkowski sum is given by $U+V:=\{u+v : u\in U, v\in V\}$ and $U^c:=\R^n\setminus U$ is the complement of $U$.
For $x_0\in \R^3$ and $r\geq 0$, we set $B_r(x_0):=\{x\in \R^3 : |x-x_0|<r\}$ as the open ball with radius $r$ and center $x_0$ (which is empty if $r=0$).
The corresponding two-dimensional open disk with radius $r$ around some $\bar{x}_0\in\R^2$ shall be denoted by $\bar{B}_r(\bar{x}_0).$
The theory of nonlocal gradients with (anisotropic) finite horizon on bounded domains requires that functions are defined on a slightly larger domain.
For any open set $O\subset \R^3$, we thus define
$$O_\d:= O + T_\d B_1(0) \quad \text{for} \quad \d=(\bar{\d},\d_3)\in [0,1]^2,$$
where $T_\d B_1(0)$ describes the ellipsoid with center $0\in\R^3$ whose principal axes are aligned with the coordinate axes and with (possibly vanishing) in-plane radius $\bar{\d}\in[0,1]$ and out-of-plane radius $\d_3\in[0,1]$, see Figure \ref{fig:ellipsoid}. Similarly, we define in the two-dimensional setting $O_{\bar{\d}}:= O + B_{\bar{\d}}(0)$ for $O\subset \R^2$.

Given a set $E\subset \R^n$, its characteristic function $\mathbbm{1}_E$ is defined via
$$\mathbbm{1}_E(x) = \begin{cases}1 & \text{ if }x\in E,\\ 0 & \text{ otherwise.}\end{cases}$$
We say that a function $f\colon \R^n\setminus\{0\} \to \R$ is radial if it depends, in fact, only on the length of its argument;
in this case, we write $f(z)= f^{\rm rad}(|z|)$ for every $z\in \R^n\setminus\{0\}$.
For an open set $U\subset \R^n$, we employ the common notation for Lebesgue and Sobolev spaces as well as for spaces of smooth functions (with compact support) with values in $\R^m$ for $m\in \N$: $L^p(U;\R^m)$, $W^{1,p}(U;\R^m)$, $C^\infty(U;\R^m)$ and $C_c^\infty(U;\R^m)$; if $m=1$ then we drop $\R^m$ in the notation.
The in-plane derivatives $(\partial_1\cdot|\partial_2\, \cdot\,)$ of $\nabla$ are abbreviated as $\bar{\nabla}$.
Moreover, we work briefly with both the Gagliardo and Bessel potential spaces $W^{s,p}(\R^n;\R^m)$ and $H^{s,p}(\R^n;\R^m)$ with fractional parameter $s\in (0,1)$.
The former are defined as
$$W^{s,p}(\R^n;\R^m):=\left\{u\in L^p(\R^n;\R^m) : [u]_{W^{s,p}(\R^n;\R^m)}:=\left(\int_{\R^n}\int_{\R^n} \frac{|u(x) - u(x')|^p}{|x-x'|^{n+sp}}\right)^{\frac{1}{p}}<\infty\right\}$$
together with the norm $\norm{u}_{W^{s,p}(\R^n;\R^m)}:=\norm{u}_{L^p(\R^n;\R^m)} + [u]_{W^{s,p}(\R^n;\R^m)}$.
The Bessel potential spaces are given by
$$H^{s,p}(\R^n;\R^m):=\left\{u\in L^p(\R^n;\R^m) : (\langle\cdot\rangle^s \widehat{u})^\vee\in L^p(\R^n;\R^m)\right\},$$
where $\langle\cdot\rangle:=\sqrt{1+|\cdot|^2}$, and $\widehat{(\cdot)}$ and $(\cdot)^\vee$ denote the Fourier transform and its inverse, respectively;
the corresponding norm is given by $\norm{u}_{H^{s,p}(\R^n;\R^m)}:=\norm{(\langle\cdot\rangle^s \widehat{u})^\vee}_{L^p(\R^n;\R^m)}$.
It is worth pointing out that there exists a constant $C>0$ such that
\begin{align}\label{eq:interpolation}
    \norm{u}_{H^{s,p}(\R^n;\R^m)}\leq C\norm{u}_{L^{p}(\R^n;\R^m)}^{1-s}\norm{u}_{W^{1,p}(\R^n;\R^m)}^s \quad \text{for every }u\in W^{1,p}(\R^n;\R^m)
\end{align}
in light of the interpolation inequality \cite[Corollary 3.19]{BGS25}.

Throughout the manuscript, we use $C>0$ for generic constants that depend on fixed data and may differ from line to line.
Any family indexed with a continuous parameter $\e>0$ refers to any sequence $(\e_k)_k$ with $\e_k \to 0$ as $k\to \infty$.
When families are indexed with two continuous parameters $\d,\e>0$, we often couple the two via $\d = \d(\e)$ and suppress the index $\d$ in the subsequent notation.

\section{Anisotropic nonlocal gradients}\label{sec:nonlocal_gradients}
Typically, nonlocal derivatives with finite horizon have been defined through kernels supported on balls of positive radius \cite{BCM23,BMS24,CKS23,KrS24}, that is,
\begin{equation}\label{eq:nonlocalgradient}
    D_\rho \phi(x):= \int_{\R^n} \frac{\phi(y)-\phi(x)}{\abs{y-x}}\frac{y-x}{\abs{y-x}}\rho(y-x)\,dy \quad \text{for $\phi \in C^{\infty}(\R^3)$.}
\end{equation}
Throughout the paper, the kernel $\rho\colon\R^n\setminus \{0\} \to [0,\infty)$ with $n=3$ is assumed to be radial and satisfying
\begin{itemize}
    \item[(H0)] $\inf_{B_{\eta_0(0)}} \rho >0$ for some $\eta_0>0$, $\supp(\rho) = \overline{B_1(0)}$ and $\displaystyle\int_{\R^n} \rho\dd z=n$.
\end{itemize}
Moreover, we impose the same technical conditions on the radial representation $\rho^{\rm rad}$ as in \cite{BMS24,CKS25}; in fact, we have marginally strengthened (H1) compared to those sources in order for this set of conditions to remain preserved in the lower-dimensional model, cf.~Lemma~\ref{le:2dgradient} below. Precisely, with $\eta_0$ as in (H0) and some $\nu >0$ and $0 < \sigma \leq \gamma <1$ we assume:
\begin{itemize}

\item[(H1)] the function $f_\rho\colon(0,\infty)\to \R, \ r \mapsto r^{n-2}\rho^{\rm rad}(r)$ is such that $r \mapsto r^\nu f_\rho(r)$ is non-increasing on $(0,\infty)$;

\item[(H2)] $f_\rho$ is smooth outside the origin and for every $k\in \N$ there exists a $C_k>0$ with
\[
\abslr{\frac{d^k}{dr^k}f_\rho(r)} \leq C_k \frac{f_\rho(r)}{r^k} \quad \text{for $r\in (0, \eta_0)$};
\]

\item[(H3)] the function $r \mapsto r^{n+\sigma-1}\rho^{\rm rad}(r)$ is almost non-increasing on $(0,\eta_0)$; 

\item[(H4)] the function $r \mapsto r^{n+\gamma-1}\rho^{\rm rad}(r)$ is almost non-decreasing on $(0,\eta_0)$.
\end{itemize}
A function $f\colon\R \to \R$ is called almost non-increasing or non-decreasing if there is some fixed constant $C>0$ so that for all $t_1<t_2$ it holds that $f(t_1) \geq Cf(t_2)$ or $f(t_1) \leq C f(t_2)$, respectively. We refer to~\cite{BMS24} for more about these conditions and examples that satisfy them. For our purposes, we are primarily interested in the following case.
\begin{example}[Truncated fractional kernel]\label{ex:fractionalkernel}
    Let $\chi\colon[0,\infty) \to [0,\infty)$ be a smooth and non-increasing function with $\supp (\chi) = [0,1]$ and $\chi(0)>0$. Then, for any $s \in (0,1)$ the truncated fractional kernel
    \[
    \rho(z):= \frac{\chi(\abs{z})}{\abs{z}^{2+s}} \quad \text{for $z \in \R^3\setminus\{0\}$},
    \]
    satisfies (H1)--(H4). Moreover, up to rescaling by a constant, (H0) is also satisfied.
\end{example}
In \cite[Proposition~3.2]{BMS24}, the authors proved that the nonlocal gradient is a type of average of the classical gradient in the sense that 
\begin{equation}\label{eq:Qrhoidentity}
    D_\rho \phi = Q_\rho * \nabla \phi = \nabla (Q_\rho *\phi) \quad \text{for all $\phi \in C^{\infty}(\R^3)$,}
\end{equation}
for
\begin{align}\label{Qrho}
	Q_\rho\colon \R^3\setminus\{0\}\to \R,\quad z\mapsto \int_{|z|}^{\infty}\frac{\rho^{\rm rad}(t)}{t}\dd t,
\end{align}
where $\rho^{\rm rad}$ is the radial representation of $\rho$.
Using (H0), one can verify that
\begin{equation}\label{eq:normalizedQ}
    \norm{Q_\rho}_{L^1(\R^3)}=1.
\end{equation}

In the analysis of thin films, however, the thickness of the body becomes vanishingly small, and one commonly introduces a rescaling in the out-of-plane direction.
This rescaling naturally produces anisotropic interaction neighborhoods, whose geometry falls outside the isotropic setting commonly treated in the literature \cite{BCM23,BMS24,CKS23,KrS24}.
The aim of this section is therefore to introduce nonlocal gradients $D_{\rho,\d}$ with an anisotropic horizon.
Specifically, we will replace the typical positive omnidirectional horizon $\delta\in (0,1]$ with two (possibly vanishing) horizons $\d=(\bar{\d},\d_3)\in[0,1]^2$, where $\bar{\d}$ describes the in-plane interaction radius in the cross-section variables $\bar{x}\in\R^2$ and $\d_3$ represents the out-of-plane horizon in $x_3\in\R$. 
In this way, the corresponding nonlocal gradients measure all interactions within ellipsoids with principal radii $\bar{\d},\bar{\d}$, and $\d_3$, see Figure \ref{fig:ellipsoid}.
This approach can also be generalized to higher dimensions and all different radii, see Remark \ref{rem:higher_dimensions} below.

We base our definition of the nonlocal gradient on the identity \eqref{eq:Qrhoidentity} with an anisotropic kernel and a suitable change of variables, which provides a simple way to define a meaningful anisotropic nonlocal gradient that is amenable to the tools already developed in \cite{BMS24,CKS25}.

\begin{lemma}[Anisotropic nonlocal gradients]\label{lem:anisotropic_nonlocal_gradients}
    Let $\ffi\in C^\infty(\R^3)$, $\d\in[0,1]^2$, and $Q_{\rho}$ is given by~\eqref{Qrho}, then the function $\Qcald \phi$ given by
    \begin{equation*}
        \Qcald \phi (x):= \int_{\R^3}Q_\rho(z) \phi(x-T_\d z)\,dz \quad \text{for $x \in \R^3$}
    \end{equation*}
    satisfies $\Qcald \phi\in C^\infty(\R^3)$. Moreover, the corresponding (nonlocal) gradient and divergence
    \begin{align}\label{Dzero}
        \Dd \coloneqq \nabla \circ \Qcald = \Qcald \circ \nabla \qand \divd \coloneqq \Div \circ \Qcald = \Qcald \circ \Div
    \end{align}
    satisfy the integration by parts formula
    \begin{align*}
        \int_{\R^3} \Dd \ffi \cdot \psi \dd x = -\int_{\R^3} \ffi\divd \psi  \dd x \quad\text{ for all }\ffi\in C_c^\infty(\R^3) \text{ and } \psi \in C_c^\infty(\R^3;\R^3).
    \end{align*}
\end{lemma}
\begin{proof}
    The smoothness of $\Qcald \phi$ is a consequence of the smoothness of $\phi$ and the fact that $Q_\rho$ is integrable and has compact support. The claims in \eqref{Dzero} follow immediately from pulling the derivative inside the integral.

    The integration by parts formula can be proved by exploiting its classical counterpart for local gradients as well as the changes of variables $\tilde{x} = x - T_\d z$ and $\tilde{z} = - z$, given that $Q_\rho$ is radial. 
\end{proof}

Now, Lemma \ref{lem:anisotropic_nonlocal_gradients} allows us to define our new nonlocal gradients in a weak sense.
We remind the reader here that $O_\d= O +T_\d B_1(0)$ for every subset $O\subset \R^3$ and $\d=(\bar{\d},\d_3)\in [0,1]^2$, where $T_\d B_1(0)$ is the ellipsoid with principal axes aligned with the coordinate axes and (possibly vanishing) principal radii $\bar{\d},\bar{\d},\d_3$, see Figure \ref{fig:ellipsoid}.

\begin{definition}[Weak anisotropic nonlocal gradients]\label{def:weak_limit_gradients}
    Let $O \subset \R^3$ be an open set, and $\d\in [0,1]^2$.
    We say that $U\in L^1_\loc(O;\R^3)$ is the weak nonlocal gradient of $u\in L^1_\loc(O_\delta)$, and write $\Dd u= U$, if
    $$\int_{O} U\cdot \psi \dd x= -\int_{O_\d} u \divd \psi \dd x\quad\text{for all }\psi\in C_c^\infty(O;\R^3).$$
    If $u\in L^1_\loc(O_\d;\R^3)$, then $\Dd u$ is defined component-wise as usual.
    Moreover, we introduce the space
    \begin{align*}
        H^{\rho,p}_{\delta}(O;\R^3) \coloneqq \{u\in L^p(\R^3;\R^3): u = 0 \text{ in  $(O_{\delta})^c$ and $(\Dd u)\restrict{O}\in L^p(O;\R^{3\times 3})$}\}
    \end{align*}
    equipped with the norm
    \begin{align*}
        \norm{u}_{\Hd{O}}\coloneqq \norm{u}_{L^p(\R^3;\R^3)} + \norm{\Dd u}_{L^p(O;\R^{3\times 3})} = \norm{u}_{L^p(O_\d;\R^3)} + \norm{\Dd u}_{L^p(O;\R^{3\times 3})}
    \end{align*}
    for $u\in \Hd{O}$ and set
    \begin{align}\label{cursed_functions}
        N^{\rho,p}_{\delta}(O;\R^3) \coloneqq \{u\in H^{\rho,p}_{\delta}(O;\R^3): (\Dd u)\restrict{O}= 0\}.
    \end{align}
\end{definition}

It is important to highlight that a vanishing nonlocal gradient does, in general, not imply constancy, see Remark~\ref{rem:N} and~\cite{KrS24} for more details.

\begin{remark}[Anisotropic nonlocal gradients of affine functions]
    If $\ffi_A(x):=Ax+b$ as an affine function with $A \in \R^{3\times 3}$ and $b \in \R^3$, then a direct computation yields that
    \[
        D_{\rho,\d}\ffi_A(x) = A \quad \text{for all $x \in \R^3$ and $\d \in [0,1]^2$,}
    \]
    which coincides with the classical gradient.
\end{remark}

\begin{remark}[Anisotropic nonlocal gradients with positive horizons]\label{rem:positive_horizon}
    Let $\delta\in (0,1]^2$. Similar to the typical radial case, the anisotropic nonlocal gradient \eqref{Dzero} can be rewritten as a convolution of the classical gradient with a suitably rescaled kernel. 
    Indeed, for $\phi \in C^\infty(\R^3)$ it holds that
    \begin{align}\label{eq:Qrhodelta}
        \Dd \ffi = \nabla\big(Q_{\rho,\d} \ast \ffi\big) = Q_{\rho,\d}\ast \nabla \ffi\quad \text{with}\quad  Q_{\rho,\d} = \det(T_{\d}^{-1})Q_{\rho}\circ T_{\d}^{-1},
    \end{align}
    cf.~\eqref{Qrho}. Note that, by a change of variables and \eqref{eq:normalizedQ}, we find that
    \begin{equation*}
       \norm{Q_{\rho,\d}}_{L^1(\R^3)}=1 \quad \text{for all $\d \in (0,1]^2$.}
    \end{equation*}
    Moreover, via a straightforward calculation, one shows that
    $$\Dd\ffi(x) = D_{\rho}\big(\ffi\circ T_{\d}\big)\big(T_{\d}^{-1}x\big) T_{\d}^{-1}\quad\text{for all }\ffi\in C^\infty(\R^3) \text{ and every }x\in\R^3,$$
    where $D_{\rho}$ is the typical common nonlocal gradient with horizon $1$ as in \eqref{eq:nonlocalgradient}.
\end{remark}

\begin{remark}[On the set $\Nd{O}$]\label{rem:N}
    a) If $O$ is connected, then the translation mechanism in Proposition \ref{prop:translation}\,($i$) below shows that $h \in \Nd{O}$ if and only if $h \in L^p(\R^3;\R^3)$ (with $h=0$ in $(O_\d)^c$) and there is some $c \in \R^3$ such that
    \[
    \Qcald h = c \quad \text{a.e.~in $O$.}
    \]
    Here, we assume that $\Qcald$ has been canonically extended to $L^p$ as in Proposition \ref{prop:translation}. Indeed, $\tilde{h}:=\Qcald h \in W^{1,p}(O;\R^3)$ satisfies $\nabla \tilde{h} = \Dd h$ and thus the characterization follows. \smallskip
    
    b) Suppose $O$ is a bounded Lipschitz domain and $\delta=(1,1)$. Then, it holds that $N^{\rho,p}(O;\R^3):=N^{\rho,p}_{\delta}(O;\R^3)$ is an infinite-dimensional space and hence, contains many non-constant functions. Indeed, this can be argued similarly as in the case for truncated fractional kernels from \cite[Proposition~3.3]{KrS24} using the more general results from \cite{BMS24}. In fact, in the case that $\rho$ is as in Example~\ref{ex:fractionalkernel} with $\chi$ locally constant around the origin, $p \in (1,\frac{2}{1-s})$ and $O$ is a $C^{1,1}$-domain, then it follows from \cite[Theorem~3.8]{KrS24} that for every boundary condition $g \in L^p(O_\d \setminus O;\R^3)$ and every $c \in \R^3$ there is a unique $h \in N^{\rho,p}(O;\R^3)$ with $h=g$ in $O_\d \setminus O$ and
    \[
    \Qcald h = c \quad \text{a.e.~in $O$.}
    \]
    This characterizes the set $N^{\rho,p}(O;\R^3)$ completely. When $p \geq \frac{2}{1-s}$, uniqueness still holds but existence may fail due to possible singularities forming at the boundary $\partial O$. This characterization does not readily extend to general kernels $\rho$, but we mention at least that uniqueness for $p \geq 2$ is still true by following the argument in \cite[Proposition~3.9]{KrS24} together with \cite[Lemma~4.3]{BMS24}. These facts will not be needed for our analysis though. \smallskip

    c) In light of Remark~\ref{rem:positive_horizon}, it immediately follows that for $\d \in (0,1]^2$ we have the identification
    \[
    \Nd{O} \cong N^{\rho,p}(T_\d^{-1}O;\R^3), \quad h \mapsto h \circ T_\d.
    \]
    Hence, the dependence on $\d$ reduces to a simple dependence on the domain. In particular, $\Nd{O}$ is still infinite-dimensional in light of part b) of this remark. \smallskip

    d) When one of the principal radii vanishes, the set $\Nd{O}$ can also be described using the dimensionally reduced nonlocal gradients. For example, if $\bar{\d}>0$ and $\d_3=0$, then in light of part a) of this remark and Lemma~\ref{le:2dgradient} below, we deduce that $h \in \Nd{O}$ if and only if
    \[
    \Qcal_{\bar{\rho},\bar{\d}}(h(\cdot,x_3)) = c \quad \text{a.e.~on $O_{x_3}$} \quad \text{with} \quad O_{x_3}:=\{\bar{x} \in \R^2 \,:\, (\bar{x},x_3) \in O\}.
    \]
    Thus, $h(\cdot,x_3)$ lies in the set $N^{\bar{\rho},p}_{\bar{\d}}(O_{x_3};\R^3)$ for a.e.~$x_3 \in \R$, with $c$ not depending on $x_3$.
\end{remark}

We are particularly interested in applications to thin films, where admissible limit deformations do not depend on the out-of-plane variable $x_3$. In this case, the anisotropic gradient above reduces to a two-dimensional (non)local gradient. 
\begin{lemma}[Dimensionally reduced (non)local gradient]\label{le:2dgradient}
    Let $\d \in [0,1]^2$, $\bar{\phi} \in C^{\infty}(\R^2;\R^3)$ and define $\phi(x) := \bar{\phi}(\bar{x})$ for $x \in \R^3$. Then, for every $x\in\R^3$ it holds that
    \[
    \Dd \phi (x) = \begin{cases}
        (\Dtwod \bar{\phi}(\bar{x}),0) &\text{if $\bar{\d} >0$,}\\
        (\bar{\nabla}\bar{\phi}(\bar{x}),0) &\text{if $\bar{\d}=0$,}
    \end{cases}
    \]
    where $\Dtwod$ is the two-dimensional in-plane analogue of \eqref{Dzero} (see also \eqref{eq:Qrhodelta}) based on a $Q_{\bar{\rho}}$ defined via \eqref{Qrho} and
    \[
    \bar{\rho}(\bar{z}):=\abs{\bar{z}}\int_{\R}\frac{\rho^{\rm rad}(\abs{\bar{z}}\sqrt{1+t^2})}{1+t^2}\,dt\quad \text{for }\bar{z} \in \R^2 \setminus \{0\}.
    \]
    Moreover, $\bar{\rho}$ satisfies the two-dimensional analogues of (H0)--(H4).
\end{lemma}
\begin{proof}
    The case $\bar{\d}=0$ follows from a direct computation. If $\bar{\d} \not =0$, we find with $\bar{Q}_\rho(\bar{z}):=\int_{\R}Q_\rho(\bar{z},z_3)\,dz_3$ for $\bar{z}\in\R^2\setminus\{0\}$ that
    \[
    \Dd \phi (x) = ((\Qbar * \bar{\nabla} \bar{\phi}) (\bar{x}),0),
    \]
    where $\Qbar$ is the two-dimensional in-plane analogue of \eqref{eq:Qrhodelta}.
    Hence, it remains to prove that $\bar{Q}_\rho = Q_{\bar{\rho}}$ in the sense that
    \[
        \bar{Q}_\rho(\bar{z}) = \int_{\abs{\bar{z}}}^{\infty}\frac{\bar{\rho}^{\rm rad}(t)}{t}\,dt \quad \text{for all } \bar{z}\in \R^2\setminus\{0\}. 
    \]
    This follows from the following computation 
    \begin{align*}
		\int_\R Q_\rho(\bar{z},z_3)\dd z_3 &= \int_\R \int_{|z|}^{\infty} \frac{\rho^{\rm rad}(t)}{t}\dd t \dd z_3 = \int_\R \int_{\sqrt{z_3^2 +|\bar{z}|^2}}^{\infty}\frac{\rho^{\rm rad}(t)}{t}\dd t \dd z_3\\
		&= \int_\R \int_{|\bar{z}|}^{\infty}  \frac{r\rho^{\rm rad}\left(\sqrt{z_3^2 + r^2}\right)}{z_3^2 + r^2}\dd r \dd z_3 =\int_{|\bar{z}|}^{\infty}\int_\R \frac{r^2\rho^{\rm rad}\left(\sqrt{\tilde{z}_3^2r^2 + r^2}\right)}{\tilde{z}_3^2r^2 + r^2}\dd \tilde{z}_3 \dd r\\
		&=\int_{|\bar{z}|}^{\infty} \int_\R \frac{\rho^{\rm rad}\left(r\sqrt{1+\tilde{z}_3^2}\right)}{1+\tilde{z}_3^2}\dd \tilde{z}_3 \dd r\ 
			= \int_{|\bar{z}|}^{\infty} \frac{\bar{\rho}^{\rm rad}(r)}{r}\dd r,
        \end{align*}
        where we have used the substitutions $t=\sqrt{z_3^2+r^2}$ and $z_3=r\tilde{z}_3$. The conditions (H0)--(H4) can be verified in a straightforward manner.
        Indeed, (H0) is a consequence of
        \[
        \int_{\R^2} \bar{\rho}(\bar{z})\,d\bar{z}= 2 \int_{\R^2} \bar{Q}_\rho(\bar{z})\,d\bar{z}=2\norm{Q_\rho}_{L^1(\R^3)}=2,
        \]     
        using polar coordinates, the behavior of $\rho$ from (H0) and (H3). 
        Condition (H1) is immediate from the definition of $\bar{\rho}$. For condition (H2), we compute that
        \[
        f_{\bar{\rho}}(r) = \int_{\R} \frac{f_\rho(r\sqrt{1+t^2})}{(1+t^2)^{3/2}}\dd t,
        \]
        so that by interchanging derivatives and integrals we infer (H2) from the same condition on $\rho$; note that (H2) only holds on $(0,\eta_0)$ while the integral also integrates beyond this point, but the extra part is a smooth and bounded function which can be accounted for in the constants $C_k$. Finally, (H3) and (H4) also follow in a straightforward manner from the definition of $\bar{\rho}$, again using that the integral for large values is smooth and can be absorbed into the constant of the almost monotonicity.
\end{proof}
\begin{remark}\label{rem:rhobaradmissible}    
    a) The previous lemma shows that $\Dtwod$ is again a nonlocal gradient that fits into the setting of nonlocal gradients of \cite{BMS24}, but in dimension 2. We will thus use the notation
    \begin{align}\label{Htwod}
        \begin{split}
            H^{\bar{\rho},p}_{\bar{\d}}(O;\R^3) &= \{u\in L^p(\R^2;\R^3) : u = 0 \text{ in $(O_\d)^c$  and }\Dtwod u \in L^p(O;\R^{3\times 2})\},\\
            N^{\bar{\rho},p}_{\bar{\d}}(O;\R^3) &= \{u\in H^{\bar{\rho},p}_{\bar{\d}}(O;\R^3) : (\Dtwod u)\restrict{O}=0\}
        \end{split}
    \end{align}
    for open sets $O\subset \R^2$. \smallskip

    b) Let $\Omega=\omega\times I$ for an open set $\omega\subset \R^2$ and a bounded open interval $I\subset \R$. Via distributional arguments as in the classical local case, one shows that if $u\in \Hd{\Omega}$ and $u(x) = \mathbbm{1}_{O_\d}(x)\bar{u}(\bar{x})$ for almost every $x\in \R^3$ and some $\bar{u}\colon \omega_\d\to \R^3$, then it holds that
    $$\mathbbm{1}_{\omega_\d}\bar{u}\in H^{\bar{\rho},p}_{\bar{\d}}(\omega;\R^3) \quad\text{with}\quad \bar{D}_{\rho,\d}u(x) = D_{\bar{\rho},\bar{\d}} \bar{u}(\bar{x}) \text{ for a.e.~}x\in\Omega,$$
    where $\bar{D}_{\rho,\d}$ describes the first two columns of $\Dd$.
    
    \smallskip

    c) If $\rho$ is as in Example~\ref{ex:fractionalkernel}, then it can be computed that
    \[
    \bar{\rho}(\bar{z}):=\frac{\bar{\chi}(\abs{\bar{z}})}{\abs{\bar{z}}^{1+s}},\ \bar{z}\in \R^2\setminus\{0\} \qquad \text{with} \qquad \bar{\chi}(r):=\int_{\R}\frac{\chi(r\sqrt{1+t^2})}{(\sqrt{1+t^2})^{4+s}}\,dt,\ r\in[0,\infty).
    \]
    Hence, $D_{\bar{\rho}}$ is again a truncated fractional gradient of order $s$.
\end{remark}

In the following, we discuss the Leibniz rule for anisotropic nonlocal gradients, beginning with two smooth functions.
This will then be used to prove the density of $C_c^\infty$-functions in $H^{\rho,p}_\d$.

\begin{lemma}[Leibniz rule for smooth functions]\label{lem:leibniz}
    Let $\delta \in [0,1]^2$ and $u \in C_c^{\infty}(\R^3;\R^3)$ and $\chi \in C^{\infty}(\R^3)$ with bounded derivative. Then, it holds that
    \[
    \Dd (\chi u) = \chi \Dd u + K_\d^{\chi}(u),
    \]
    where
    \begin{align}\label{K_chi}
    K^{\chi}_\d(u)(x):=\int_{\R^3} Q_\rho(z)\left[(\chi(x-T_\d z)-\chi(x))\nabla u(x-T_\d z) + u(x-T_\d z)\nabla \chi(x-T_\d z)\right]\,dz.
    \end{align}
    Moreover, there exists a $C>0$ independent of $\d$ such that
    \begin{equation}\label{eq:leibnizest1}
    \norm{K^{\chi}_\d(u)(e_1 | e_2)}_{L^p(\R^3)} \leq \begin{cases} \displaystyle\frac{C}{\bar{\d}}\norm{\nabla \chi}_{L^{\infty}(\R^3)}\norm{u}_{L^{p}(\R^3)} &\text{if $\bar{\d} > 0$,}\\
    C(\norm{\bar{\nabla} \chi}_{L^{\infty}(\R^3)}\norm{u}_{L^{p}(\R^3)}+\norm{\partial_3\chi}_{L^{\infty}(\R^3)}\norm{\bar{\nabla} u}_{L^p(\R^3)}) &\text{if $\bar{\d}=0$,}
    \end{cases}
    \end{equation}
    and
    \begin{equation}\label{eq:leibnizest2}
    \norm{K^{\chi}_\d(u)e_3}_{L^p(\R^3)} \leq \begin{cases} \displaystyle\frac{C}{\d_3}\norm{\nabla \chi}_{L^{\infty}(\R^3)}\norm{u}_{L^{p}(\R^3)} &\text{if $\d_3 > 0$,}\\
    C(\norm{\partial_3 \chi}_{L^{\infty}(\R^3)}\norm{u}_{L^{p}(\R^3)}+\norm{\bar{\nabla}\chi}_{L^{\infty}(\R^3)}\norm{\partial_3 u}_{L^p(\R^3)}) &\text{if $\d_3=0$.}
    \end{cases}
    \end{equation}
\end{lemma}
\begin{proof}
    Using the definition of $\Dd$ in \eqref{Dzero} and the classical Leibniz rule, we find that
    \begin{align*}
    \Dd (\chi u) &= \Qcald (\nabla(\chi u)) = \Qcald(\chi \nabla u+u\nabla \chi) \\
    &= \chi\Dd u + (\Qcald (\chi \nabla u)-\chi\Qcald (\nabla u)+\Qcald(u \nabla \chi)) = \chi \Dd u + K^{\chi}_\d(u).
    \end{align*}
    For the estimates, we only prove \eqref{eq:leibnizest2} since the proof of \eqref{eq:leibnizest1} is entirely analogous. Suppose that $\d_3>0$, then we can use integration by parts on the first term of $K^{\chi}_\d(u)e_3$ in \eqref{K_chi} to find
    \begin{align*}
        &\int_{\R^3} Q_\rho(z)(\chi(x-T_\d z)-\chi(x))\partial_3 u(x-T_\d z)\,dz \\
        &= \int_{\R^3} \frac{1}{\d_3}\partial_3 Q_\rho(z)(\chi(x-T_\d z)-\chi(x))u(x-T_\d z) - Q_\rho(z)\partial_3\chi(x-T_\d z)u(x-T_\d z)\,dz.
    \end{align*}
    This second term cancels with the second term of $K^{\chi}_\d(u)$, so we find that
    \begin{align}\label{K_chi_integrated}
        K^\chi_\d(u)(x)e_3&= \int_{\R^3} \frac{1}{\d_3}\partial_3 Q_\rho(z)(\chi(x-T_\d z)-\chi(x))u(x-T_\d z)\,dz \nonumber\\
        &=\int_{\R^3} \frac{1}{\d_3}\frac{-z_3\rho(z)}{\abs{z}^2}(\chi(x-T_\d z)-\chi(x))u(x-T_\d z)\,dz,
    \end{align}
    where the second line uses the definition of $Q_\rho$ from \eqref{Qrho}. Using the Lipschitz continuity of $\chi$, we find that
    \begin{align*}
         \abs{K^\chi_\d(u)(x)e_3} &\leq \frac{1}{\delta_3}\norm{\nabla \chi}_{L^{\infty}(\R^3)} \int_{\R^3} \frac{\abs{z_3}\abs{T_\d z}}{\abs{z}^2}\rho(z)\abs{u(x-T_\d z)}\,dz \\
         &\leq \frac{1}{\delta_3}\norm{\nabla \chi}_{L^{\infty}(\R^3)} \int_{\R^3}\rho(z)\abs{u(x-T_\d z)}\,dz,
    \end{align*}
     since $\abs{z_3}\abs{T_\d z}\leq \abs{z}^2$. We conclude the first estimate in \eqref{eq:leibnizest2} after applying Minkowski's integral inequality and a change of variables. 

    For the case $\d_3 =0$, we directly use the definition of $K^\chi_\d(u)$ to find that
    \begin{align*}
        \abs{K^\chi_\d(u)(x)e_3} \leq \norm{\bar{\nabla}\chi}_{L^{\infty}(\R^3)}\int_{\R^3}Q_\rho(z)\abs{T_\d z}\abs{\partial_3 u(x-T_\d z)}\,dz + \norm{\partial_3 \chi}_{L^{\infty}(\R^3)}\int_{\R^3}Q_\rho(z)\abs{u(x-T_\d z)}\,dz,
    \end{align*}
    where we have used that $(x-T_\d z)\cdot e_3 = x_3$ for the first term in order to only have to use a Lipschitz estimate in the first two variables. Using that $\abs{T_\d z} \leq \abs{z}$ together with the fact that $Q_\rho$ and $\abs{\cdot}Q_\rho$ are integrable, we can conclude as before with Minkowski's integral inequality.
\end{proof}

\begin{remark}[Some properties of $K^\chi_\delta(u)$]\label{rem:Kchi}
    a) Similar to \eqref{K_chi_integrated}, we write for reference the expression for the first two components of $K^\chi_\d(u)$ when $\bar{\delta}>0$:
    \begin{equation}\label{K_chi12}
         K^\chi_\d(u)(x)(e_1 | e_2)=\int_{\R^3} \frac{1}{\bar{\delta}}\frac{-\bar{z}\rho(z)}{\abs{z}^2}(\chi(x-T_\d z)-\chi(x))u(x-T_\d z)\,dz.
    \end{equation}

    b) The quantity $K^\chi_\d(u)$ in \eqref{K_chi}, or an alternative formulation, exists also for less regular functions $u\in L^p(\R^3;\R^3)$ with suitable properties.
    
    For example, if $\bar{\d}>0$ and $\d_3 >0$, then we use the formulations \eqref{K_chi_integrated} and \eqref{K_chi12} instead of \eqref{K_chi}, which exist almost everywhere without any additional assumptions on $u$.
    If only one component of $\d$ vanishes, say $\d_3 = 0$, and $\bar{\nabla} \chi = 0$, then \eqref{K_chi} does not require any extra regularity for $u$ as it does not see $\nabla u$ anymore since it is multiplied by zero.
    On the other hand, if $\d_3 =0$ and $\bar{\nabla}\chi \neq 0$, then the last column of \eqref{K_chi} exists almost everywhere if $u$ has at least an integrable weak $x_3$-derivative; 
    for the first two columns, we employ the form \eqref{K_chi12} to avoid derivatives of $u$.
    One argues analogously for $\bar{\d}=0$ and $\d_3>0$.
    
    These case studies and the different representations of $K^\chi_\d$ are a central part of the next lemma.    
\end{remark}

We now deal with the Leibniz rule of smooth functions and ones in $\Hd{\Omega}$ for open and bounded sets $\Omega\subset\R^3$. 
The proof essentially follows the structure of \cite[Lemma~12]{CKS23} and uses the calculations and estimates in Lemma \ref{lem:leibniz}.
\begin{lemma}[Leibniz rule in $H^{\rho,p}_\d$]\label{lem:product}
    Let $\Omega,\Omega'\subset \R^3$ be open and bounded, and $\d\in [0,1]^2$. 
    Given $u\in \Hd{\Omega}$ and $\chi \in C^ \infty(\R^3)$, their product satisfies $\chi u \in L^p(\Omega'_{\delta};\R^3)$ and $\Dd(\chi u) \in L^p(\Omega';\R^{3\times 3})$, i.e., $\chi u \in \Hd{\Omega'}$ if 
    \begin{align}\label{support}
        (\Omega'\setminus\Omega)\cap\supp(\chi)  = \emptyset
    \end{align}
    and if one of the following conditions is satisfied
    \begin{itemize}
        \item[a)] $\bar{\d}>0$ and $\d_3 >0$;
        \item[b)] $\bar{\d}>0$ and $\d_3=0$ and
        \begin{itemize}
            \item[$(i)$] $\bar{\nabla}\chi =0$ or
            \item[$(ii)$] the weak derivative $\partial_3u$ exists in $L^p(\Omega_\delta;\R^3)$;
        \end{itemize}
        \item[c)] $\bar{\d}=0$ and $\d_3 >0$ and
        \begin{itemize}
            \item[$(i)$] $\partial_3 \chi =0$ or
            \item[$(ii)$] the weak derivative $\bar{\nabla}u$ exists in $L^p(\Omega_\delta;\R^{3\times 2})$.
        \end{itemize}
    \end{itemize}
\end{lemma}
\begin{proof}
    We prove here only the cases a), b) $i)$ and b) $ii)$. The case c) can be handled analogously to b). Let $\ffi\in C_c^\infty(\Omega';\R^3)$ be given, then $\chi\ffi\in C_c^\infty(\Omega;\R^3)$ since \eqref{support} and it holds that $\divd (\chi\ffi) = \chi \divd \ffi + L_\d^{\chi}(\ffi)$ with
    \begin{align}\label{L_chi}
        L^{\chi}_\d(\ffi)(x):=\sum_{i=1}^3K_\d^\chi(\ffi_i)e_i
    \end{align}
    where $K_\d^\chi(\ffi_i)$ is given as in \eqref{K_chi} for scalar-valued functions. 
    Now we distinguish all the different cases.
    \smallskip
    
    a) If $\bar{\d},\d_3>0$, then we use for $K_\d^\chi(\ffi_i)e_i$ the representation as in \eqref{K_chi_integrated} and \eqref{K_chi12}, which emerges through integration by parts.
    Using Fubini's theorem, a suitable change of coordinates, together with $\supp \phi \subset \Omega'$ and $u\in \Hd{\Omega}$, we obtain that
    \begin{align*}
        \int_{\Omega'_\d}(\chi u)\divd \ffi \dd x &= \int_{\Omega'_\d}u \big(\divd(\chi\ffi) - L^\chi_\d(\ffi)\big) \dd x 
        = -\int_{\Omega}\Dd u (\chi\ffi)\dd x - \sum_{i=1}^3 \int_{\Omega_\d'}uK_\d^\chi(\ffi_i)e_i\dd x\\
        &=-\int_{\Omega'}\Dd u (\chi\ffi)\dd x - \int_{\Omega'} \ffi_1K^\chi_\d(u)e_1 + \ffi_2K^\chi_\d(u)e_2 +\ffi_3K_\d^\chi(u)e_3\dd x\\
        &=-\int_{\Omega'}\big(\chi\Dd u +  K_\d^\chi(u)\big)\ffi\dd x.
    \end{align*}
    The quantities $K^\chi_\d(u)e_i$ for all $i\in\{1,2,3\}$ are written in the form of \eqref{K_chi_integrated} and \eqref{K_chi12} so that no weak derivatives of $u$ are needed, see Remark \ref{rem:Kchi}.
    This proves that 
    $$\Dd (\chi u) = \chi\Dd u +  K_\d^\chi(u)$$
    on $\Omega'$.
    The first term is in $L^p(\Omega';\R^{3\times 3})$ due to \eqref{support} and the second one can be estimated as in Lemma~\ref{lem:leibniz}. 
    \smallskip
    
    b) $i)$ If now $\d_3=0$ and $\partial_1\chi=\partial_2\chi=0$, then \eqref{L_chi} turns into
    \begin{align*}
        L^{\chi}_\d(\ffi)(x):=\int_{\R^3} Q_\rho(z)\partial_3\chi(x)\ffi_3(x-T_\d z)\dd z =\int_{\R^3} Q_\rho(z)\nabla\chi(x)\cdot\ffi(x-T_\d z)\dd z
    \end{align*}
    so that
    \begin{align*}
        \int_{\Omega'_\d}(\chi u)\divd \ffi \dd x &= \int_{\Omega'_\d}u \big(\divd(\chi\ffi) - L^\chi_\d(\ffi)\big) \dd x 
        = -\int_{\Omega}\Dd u (\chi\ffi)\dd x - \int_{\Omega'_\d}uL^\chi_\d(\ffi) \dd x\\
        &=-\int_{\Omega}\Dd u (\chi\ffi)\dd x + \int_{\Omega'}\Big(\int_{\R^3}Q_\rho(z)u(x-T_\d z)\nabla\chi(x)\dd z\Big)\ffi(x)\dd x\\
        &=-\int_{\Omega'}\Big(\chi(x)\Dd u(x) + \int_{\R^3}Q_\rho(z)u(x-T_\d z)\nabla\chi(x)\dd z \Big)\ffi(x)\dd x.
    \end{align*}
    In this case, we find 
    $$\Dd (\chi u) = \chi\Dd u + \nabla \chi\cdot\Qcald u$$
    on $\Omega'$. The first term is as before and the second one can be estimated analogously to Lemma \ref{lem:leibniz}.
    \smallskip
    
    $ii)$ On the other hand, if $\d_3=0$ and the weak derivative $\partial_3u$ exists and lies in $L^p(\Omega_\d;\R^3)$, then we use \eqref{L_chi} with $K_\d^\chi(\ffi_i)$ as in \eqref{K_chi12} for the cases $i=1$ and $i=2$, while we use \eqref{K_chi} for $i=3$.
    We repeat a similar argument as before to obtain
    \begin{align*}
        \int_{\Omega'_\d}(\chi u)\divd \ffi \dd x &= \int_{\Omega'_\d}u \big(\divd(\chi\ffi) - L^\chi_\d(\ffi)\big) \dd x 
        = -\int_{\Omega'}\Dd u (\chi\ffi)\dd x - \int_{\Omega'_\d}uL^\chi_\d(\ffi) \dd x\\
        &=-\int_{\Omega'}(\chi\Dd u) \ffi\dd x - \sum_{i=1}^3 \int_{\Omega_\d'}uK_\d^\chi(\ffi_i)e_i\dd x=-\int_{\Omega'}\big(\chi\Dd u + K^\chi_\d(u)\big) \ffi\dd x.
    \end{align*}
    Note that the first and second column of $K_\d^\chi(u)$ are written in the form \eqref{K_chi12} as to avoid weak differentiability of $u$ in the in-plane variables. For the third column, we exploit that \eqref{K_chi} also holds with $\partial_3 u$ as the weak derivative, cf.~Remark \ref{rem:Kchi}. All terms above can now be estimated just like in Lemma \ref{lem:leibniz}.
\end{proof}
\begin{remark}
    We note that the additional assumptions in Lemma~\ref{lem:product} b) and c) are not merely of a technical nature. For example, in the setting b) and with $\Omega=\omega \times I$, we can consider a non-constant function $\bar{v} \in N^{\bar{\rho},p}_{\bar{\d}}(\omega;\R^3)$ with $\Qcaltwod \bar{v}(\bar{x})=0$, which exists in light of Remark~\ref{rem:N}. Then, for any $w \in L^p(\R)$ it holds that $u(x):=\mathbbm{1}_{\O_\d}(x)v(\bar{x})w(x_3) \in \Nd{\O}$. Indeed, one has that
    \[
    \Qcald u(x) = w(x_3)\Qcaltwod \bar{v}(\bar{x}) =0 \quad \text{for a.e.~$x \in \O$.}
    \]
    However, once we multiply $\bar{v}$ by a function $\bar{\chi} \in C^{\infty}(\R^2)$, it may happen that $\Qcaltwod(\bar{\chi}\bar{v}) \not =0$. In that case, and if $w|_{I} \not \in W^{1,p}(I)$, we find that
    \[
    \Qcald(\bar{\chi}u)(x) = w(x_3)\Qcaltwod(\bar{\chi}\bar{v})(\bar{x}),
    \]
    is not a $W^{1,p}(\O;\R^3)$ function. In light of Proposition \ref{prop:translation}\,($i$) below, this shows $\bar{\chi} u \not \in \Hd{\O}$.
\end{remark}

Having established the Leibniz rule, we turn to the density of $C_c^\infty(\R^3;\R^3)$ in $\Hd{\O}$ on cylindrical bounded Lipschitz domains $\Omega$. 
After a regularization in the out-of-plane variable, the proof follows the same arguments as in \cite[Theorem~1 and Lemma~13]{CKS23}.

\begin{proposition}[Density on bounded Lipschitz domains]\label{prop:density}
    Let $\Omega=\omega\times I$ with $\omega\subset\R^2$ a bounded Lipschitz domain, $I\subset \R$ an open interval and let $\d\in[0,1]^2$. For each $u \in \Hd{\O}$, there exists a sequence $(\phi_k)_k \subset C_c^{\infty}(\R^3;\R^3)$ such that $\phi_k \to u$ in $\Hd{\Omega}$, that is, $\phi_k \to u$ in $L^p(\Omega_{\d};\R^3)$ and $\Dd \phi_k \to \Dd u$ in $L^p(\Omega;\R^{3\times 3})$.
\end{proposition}
\begin{proof}
    Without loss of generality, we assume that $I=(0,1)$.
    
    The case $\d \in (0,1]^2$ can be argued as in \cite[Theorem~1]{CKS23} using the Leibniz rule in Lemma~\ref{lem:leibniz} and \ref{lem:product} a) and does not even require the Cartesian product structure of $\Omega$. 
    The case $\d=(0,0)$ is the classical density result for Sobolev spaces. 
    It remains to prove the case $\bar{\d}>0$ and $\d_3=0$ and the case $\bar{\d}=0$ and $\d_3 >0$. 
    Since the proof strategy is analogous for both cases, we only detail the former case and split the proof into two steps.\smallskip

    \textit{Step 1: Approximation by functions with higher regularity in the third variable.} In this step, we prove that there is a sequence $(u_k)_k \subset L^p(\R^3;\R^{3\times 3})$ such that $u_k \to u$ in $L^p(\R^3;\R^3)$ and $\Dd u_k \to \Dd u$ in $L^p(\Omega;\R^3)$ as $k \to \infty$ and the weak derivatives $\partial_3^{n} u_k$ exist for all $n,k \in \N$ and lie in $L^p(\R^3;\R^3)$. 
    The strategy is to use a standard mollification and partition of unity approach in the $x_3$-direction.
    
    Consider two cut-off functions $\chi_1,\chi_2 \in C_c^{\infty}(\R;[0,1])$ such that $\chi_1+\chi_2=1$ on $(0,1)$ and $\supp(\chi_1) \subset (-\infty,3/4)$ and $\supp(\chi_2) \subset (1/4,\infty)$. 
    We then define $u_1(x)=\chi_1(x_3) u(x)$ and $u_2(x):=\chi_2(x_3) u(x)$ for a.e.~$x\in \R^3$ so that $u_1 + u_2 = u$ on $\Omega_\d = \omega_{\bar\d} \times (0,1)$ (note that $\d_3=0$) and $u_1=u_2=0$ on $(\Omega_\d)^c$.
    By Lemma~\ref{lem:product} b) $i)$, we find that $u_1 \in H^{\rho,p}_\d(\omega \times (0,\infty);\R^3)$ and $u_2 \in H^{\rho,p}_\d(\omega\times (-\infty,1);\R^3)$ since $\chi_1,\chi_2$ only depend on the third variable. If we denote translation by a vector $\zeta \in \R^3$ by $\tau_{\zeta}(v):=v(\cdot - \zeta)$, then we can define the sequence
    \[
    \tilde{u}_k: = \tau_{-e_3/k}(u_1)+\tau_{e_3/k}(u_2) \in \Hd{\omega\times (-\tfrac{1}{k},1+\tfrac{1}{k})}
    \]
    which satisfies $\tilde{u}_k \to u$ in $L^p(\R^3;\R^3)$ and by translation invariance of the nonlocal gradient
    \[
    \Dd \tilde{u}_k = \tau_{-e_3/k}(\Dd u_1)+\tau_{e_3/k}(\Dd u_2) \to \Dd u_1 + \Dd u_2 = \Dd (u_1+u_2)=\Dd u \ \text{in $L^p(\Omega;\R^{3\times3})$}
    \]
    as $k \to \infty$.
    Therefore, if $\epsilon < \frac1k$ and $\eta \in C_c^{\infty}(\R;[0,1])$ with $\supp \eta \subset B_1(0)$ and $\int_\R \eta\dd z_3 =1$, we can define 
    \[
    u_{k,\epsilon}(x) := \int_{\R} \eta(z_3)\tau_{\epsilon z_3e_3}(\tilde{u}_k)(x)\,dz_3 \quad \text{for a.e.~$x \in \R^3$},
    \]
    which is the function $\tilde{u}_k$ mollified only in the $x_3$-direction. In particular, it satisfies $u_{k,\e}\to \tilde{u}_k$ in $L^p(
    \R^3;\R^3)$ as $\e\to 0$ as well as the regularity requirement in the third variable due to Fubini, changes of variables and integration by parts. Additionally, on $\Omega$ the nonlocal gradient of $u_{k,\epsilon}$ is also given by mollification as
    \[
    \Dd u_{k,\epsilon}(x) = \int_{\R} \eta(z_3)\tau_{\epsilon z_3e_3}(\Dd \tilde{u}_k)(x)\,dz_3 \quad \text{for a.e.~$x \in \Omega$}.
    \]
    This shows that $\Dd u_{k,\epsilon} \to \Dd \tilde{u}_k$ in $L^p(\Omega;\R^{3\times 3})$ as $\epsilon \to 0$. Hence, we find that a suitable diagonal sequence $u_k:=u_{k,\epsilon_k} \in \Hd{\Omega}$ satisfies all the desired properties.  \smallskip

    \textit{Step 2: Smoothing the remaining variables.} 
    If $\chi \in C_c^{\infty}(\R^2)$, then we can use Lemma~\ref{lem:product} b) $ii)$ to conclude that $\chi u_k\colon x\mapsto \chi(\bar{x}) u_k(x)$ satisfies $\chi u_k \in \Hd{\omega'\times (0,1)}$ as long as $(\omega'\setminus \omega) \cap \supp(\chi) = \emptyset$; here, we crucially used that the third partial derivative of $u_k$ exists in the weak sense.
    Hence, we can completely follow the proof of \cite[Theorem~1 and Lemma~13]{CKS23} using partition of unity, translation and mollification -- albeit only in the first two variables -- in order to be able to approximate each $u_k$ with a sequence of  functions $(u_{k,j})_j$ whose weak in-plane gradients exist of all orders and are in $L^p$. 
    Combining this with the fact that the weak derivative $\partial_3^n u_k$ exists for all $n\in\N$ yields that $u_{k,j}\in W^{n,p}(\R^3;\R^3)$ for all $n\in\N$ and all $u_{k,j}$ vanish outside of a compact set (uniformly in $k,j$).
    Standard Sobolev embedding results yield that all of these functions are actually smooth. A diagonalization argument then finishes the proof.
\end{proof}

Density on the full space $\R^3$ is even easier to prove and simply follows from mollification and a cut-off argument. 

\begin{proposition}[Density on the full space]\label{prop:density_full_space}
    Let $\d\in[0,1]^2$ be given. For each $u \in \Hd{\R^3}$, there exists a sequence $(\phi_k)_k \subset C_c^{\infty}(\R^3;\R^3)$ such that $\phi_k \to u$ in $\Hd{\R^3}$.
\end{proposition}
\begin{proof}
    Let $(\eta_j)_j$ be a standard sequence of mollifiers and set $u_j:=u\ast\eta_j$ for $u\in \Hd{\R^3}.$ From Fubini's theorem, it is apparent that
    $$\Qcald u_j = \Qcald(u\ast\eta_j) = (\Qcald u)\ast\eta_j \qand \Dd u_j = (\Dd u)\ast \eta_j.$$
    As $j\to \infty$, this gives $u_j\to u$ in $\Hd{\R^3}$.

    Now, fix $j\in\N$ and choose some $\chi\in C_c^\infty(\R^3)$ with $\chi = 1$ on $B_1(0)$ and $\chi = 0$ on $\R^3\setminus B_2(0).$ We set $\chi_k=\chi(\frac{\cdot}{k})$ and define $u_{j,k}:=\chi_ku_j\in C_c^\infty(\R^3;\R^3).$
    Obviously, it holds that $u_{j,k} \to u_j$ in $L^p(\R^3;\R^3)$ as $k\to \infty$. We now deal with the convergence of the nonlocal gradient. To this end, we use the classical product rule and obtain
    $$\Dd u_{j,k} - \Dd u_j = \Qcald\big((\chi_k-1)\nabla u_j + (\nabla \chi_k)u_j\big).$$
    From Minkowski's integral inequality, $\norm{Q_\rho}_{L^1(\R^3)}=1$ (cf.~\eqref{eq:normalizedQ}) and a change of variables, we conclude that
    $\norm{\Qcald \psi}_{L^p(\R^3;\R^3)}\leq \norm{\psi}_{L^p(\R^3;\R^3)}$ for every $\psi\in C^\infty(\R^3;\R^3)\cap L^p(\R^3;\R^3)$. This furnishes the estimate
    \begin{align*}
        \norm{\Dd u_{j,k} - \Dd u_j}_{L^p(\R^3;\R^{3\times3})}&\leq \norm{(\chi_k-1)\nabla u_j}_{L^p(\R^3;\R^{3\times3})} + \norm{(\nabla\chi_k)u_j}_{L^p(\R^3;\R^{3\times3})}\\
        &\leq \norm{(\chi_k-1)\nabla u_j}_{L^p(\R^3;\R^{3\times3})} + \frac{C}{k}\norm{u_j}_{L^p(\R^3;\R^{3\times3})},
    \end{align*}
    for a constant $C>0$ independent of $j$ and $k$, which proves the convergence of the gradients as $k \to \infty$. A standard diagonal argument concludes the proof.
\end{proof}

One of the main aspects when dealing with nonlocal gradients of finite horizon is a translation mechanism, which allows one to connect the nonlocal theory with the classical one. While this relation works perfectly on the full space, when restricting to bounded domains it only works up to functions with vanishing nonlocal gradient. We prove that the translation mechanism, first introduced in \cite{CKS23}, works for all radii $\delta\in [0,1]^2$ with constants independent of the choice of $\delta$. In particular, the result holds for sequences $(\delta^\eps)_\eps$ with $\delta^\eps := (\bar{\delta},\frac{\delta_3}{\eps}) \to \delta^0\in [0,1]^2$ where $\delta = \delta(\eps)\in [0,1]^2$ and $\eps\to 0$, as is its main application in the proof of Theorem \ref{th:state} below.

\begin{proposition}[Translation mechanism]\label{prop:translation}
    Let $\d \in [0,1]^2$ and $O\subset\R^3$ be open, then the following statements hold: 
    \begin{itemize}
    \item[(i)] The linear operator $\Qcald\colon \Hd{O} \to W^{1,p}(O;\R^3),\ u\mapsto \int_{\R^3}Q_\rho(z)u(\cdot -T_\d z)\dd z$
    satisfies 
    $$\nabla(\Qcald u) = \Dd u\qand\norm{\Qcald u}_{W^{1,p}(O;\R^3)} \leq \norm{u}_{\Hd{O}}$$
    for all $u\in \Hd{O}$.
    \smallskip

    \item[(ii)] There is a linear operator $\Pcald\colon W^{1,p}(\R^3;\R^3)\to\Hd{\R^3}$ such that 
    $$\Pcald=\Qcald^{-1}\qand \norm{\Pcald u}_{\Hd{\R^3}} \leq C \norm{u}_{W^{1,p}(\R^3;\R^3)}$$
    for every $u\in W^{1,p}(\R^3;\R^3)$ and a constant $C>0$ that is independent of $\d$. Moreover, for every $\lambda \in (0,\sigma)$ with $\sigma$ from (H3), the map $\Pcald\colon W^{1,p}(\R^3;\R^3) \to H^{\lambda,p}(\R^3;\R^3)$ is bounded with an operator norm independent of $\d$. 
    \smallskip
    
    \item[(iii)] If $O$ has Lipschitz boundary, then there is an isomorphism
    $$\Hd{O} / \Nd{O} \cong W^{1,p}(O;\R^3)/{\rm Const}(O;\R^3)$$
    with operator norms uniformly bounded in $\d$, where ${\rm Const}(O;\R^3)$ is the space of constant functions on $O$ with values in $\R^3$ and $\Nd{O}$ as in \eqref{cursed_functions}.
    \end{itemize}
\end{proposition}
\begin{proof}
    \textit{(i)} Let $u\in \Hd{O}$, then $v\coloneqq \Qcald u \in L^{p}(O;\R^3)$ with $\norm{v}_{L^p(O;\R^3)} \leq \norm{u}_{L^p(\R^3;\R^3)}$, which can be established exactly as in the proof of Proposition \ref{prop:density_full_space}.
    Moreover, if $\psi \in C_c^\infty(O;\R^3)$, then it holds that
    \begin{align*}
        \int_O \Dd u\cdot \psi \dd x = - \int_{O_{\d}}u\cdot\divd \psi \dd x = - \int_{O_{\d}} u \cdot \Qcald (\Div \psi) \dd x = -\int_O v \cdot \Div \psi \dd x,
    \end{align*}
    where we have used Lemma \ref{lem:anisotropic_nonlocal_gradients}, Fubini's theorem, and a suitable change of variables. 
    This furnishes $\Dd u = \nabla v = \nabla (\Qcald u)$ on $O$, which finishes the proof of ($i$).
    \medskip
    
    \textit{(ii)} \textit{Step 1: Construction of $\Pcald$.}  
    First, we recall that when $\d=(1,1)$, there exists a bounded linear operator
    \[
    \Pcal_\rho\colon W^{1,p}(\R^3;\R^3) \to H^{\rho,p}(\R^3;\R^3),
    \]
    which is the inverse of $\Qcal_\rho:=\Qcal_{\rho,(1,1)}$, see~\cite[Lemma~2.12]{CKS25}. Moreover, for $\phi \in C_c^{\infty}(\R^3;\R^3)$, it holds by \cite[Equation~(31)]{CDI25} that
    \begin{align}\label{kernel_fundamental_theorem}
    \Pcal_\rho \phi (x) = \int_{\R^3} V_\rho(x-y) \cdot \nabla \phi(y)\,dy \quad \text{for all $x \in \R^3$,}
    \end{align}
    where $V_\rho \in C^{\infty}(\R^3\setminus\{0\};\R^3)$ is the radially symmetric kernel from the nonlocal fundamental theorem of calculus (\cite[Theorem~5.2]{BMS24}), which satisfies
    \begin{equation}\label{eq:Vrhobounds}
        |V_\rho(z)|+|z||\nabla V_\rho(z)| \leqslant C \max\left\{\frac{1}{|z|^{3-\sigma}},\frac{1}{|z|^{2}}\right\} \quad\text{for all $z \in \R^3\setminus\{0\}$},
    \end{equation}
    and $\Div V_\rho$ is a radial function which agrees with a Schwartz function on $B_1(0)^c$, see~\cite[Lemma~4.8]{CDI25}. We claim that
    \begin{equation}\label{eq:claim}
        \Pcal_\rho \phi (x) = \int_{\R^3}(\phi(x+h)-\phi(x))\Div V_\rho (h)\,dh\quad\text{for all $x\in \R^3$.}
    \end{equation}
    Indeed, we can compute that
    \begin{align*}
        \Pcal_\rho \phi (x) &= \lim_{r \downarrow 0}\int_{B_r(x)^c} V_\rho(x-y) \cdot \nabla_y (\phi(y)-\phi(x))\,dy \\
        &= \lim_{r \downarrow 0} -\int_{B_r(x)^c} \Div_y (V_\rho(x-y))   (\phi(y)-\phi(x))\,dy + \int_{\partial B_r(x)}V_\rho(x-y)(\phi(y)-\phi(x))\,d\Hcal^{2}(y)\\
        &=\int_{\R^3}(\phi(y)-\phi(x))\Div V_\rho (y-x)\,dy,
    \end{align*}
    where we have used integration by parts in the first equality, and \eqref{eq:Vrhobounds} together with the fact that $\phi$ is Lipschitz to find that the boundary integral vanishes in the limit $r \downarrow 0$.
    For $\d \in [0,1]^2$, we then define 
    \begin{align}\label{Pcal_3d}
        \Pcald \phi (x):=\int_{\R^3}(\phi(x+ T_{\d}h)-\phi(x))\Div V_\rho(h)\,dh\quad \quad \text{for $\phi \in C_c^{\infty}(\R^3;\R^3)$}
    \end{align}
    as the candidate for the inverse of $\Qcald$.\smallskip

    \textit{Step 2: Invertibility and estimates of $\Pcald$.}  
    In the following, we first discuss the case $\d\in(0,1]^2$, so that the inverse operator $T_\d^{-1}$ exists.
    Using this rescaling, it is straightforward to prove that $(\Qcald\circ\Pcald)\phi = \phi$ for every $\phi \in C_c^\infty(\R^3;\R^3)$ by using the equivalent formulations \eqref{kernel_fundamental_theorem} and \eqref{eq:claim} as well as the auxiliary function $\psi:=\phi\circ T_\d$. The reverse identity can be handled analogously after pointing out that $\Qcald\phi \in C_c^\infty(\R^3;\R^3)$ as well.
    
    We now prove uniform estimates of the $L^p$-norm of $\Pcald \phi$ and its nonlocal gradient.
    To this aim, we use Minkowski's integral inequality and $|T_\d h| \leq |h|$ for all $h\in \R^3$, to find for all $\phi \in C_c^{\infty}(\R^3;\R^3)$ that
    \begin{align}\label{Pcaldelta_L^p}
        \norm{\Pcald \phi}_{L^p(\R^3;\R^3)} &\leq \int_{\R^3}\norm{\phi(\cdot+T_{\d}h)-\phi(\cdot)}_{L^p(\R^3;\R^3)}|\Div V_\rho(h)|\,dh \nonumber\\
        &\leq C\norm{\phi}_{W^{1,p}(\R^3;\R^3)} \int_{\R^3}\abs{T_{\d}h}\abs{\Div V_\rho (h)}\,dh \leq C\norm{\phi}_{W^{1,p}(\R^3;\R^3)}\int_{\R^3}\abs{h}\abs{\Div V_\rho (h)}\,dh.
    \end{align}
    The final integral satisfies
    \begin{align*}
       \int_{\R^3}\abs{h}\abs{\Div V_\rho (h)}\,dh &= \int_{B_1(0)}\abs{h}\abs{\Div V_\rho (h)}\,dh + \int_{B_1(0)^c}\abs{h}\abs{\Div V_\rho (h)}\,dh \\
       &\leq C\Big(\int_{B_1(0)}\frac{1}{|h|^{3-\sigma}}\dd h +\int_{B_1(0)^c}\frac{1}{\abs{h}^{4}}\,dh\Big) < \infty,
    \end{align*}
    where we have used \eqref{eq:Vrhobounds} and that $\Div V_\rho$ is a radial Schwartz function on $B_1(0)^c$.
    Moreover, it holds that $(\Qcald\circ\Pcald)\phi=\phi$ by design, so that
    $(\Dd \circ \Pcald) \phi = (\nabla \circ \Qcald\circ \Pcald)\phi =  \nabla\phi$, which allows us to immediately deduce
    \begin{align}\label{Pcaldelta_full_estimate}
        \norm{\Pcald \ffi}_{\Hd{\R^3}} \leq C \norm{\ffi}_{W^{1,p}(\R^3;\R^3)}
    \end{align}
    for all $\phi\in C_c^\infty(\R^3;\R^3).$
    Hence, $\Pcald$ extends to a linear and bounded operator from $W^{1,p}(\R^3;\R^3)$ to $\Hd{\R^3}$ with operator norm bounded uniformly in $\d$ and $\Pcald = \Qcald^{-1}$.
    Note that we have exploited here both the density of $C_c^\infty(\R^3;\R^3)$ in $W^{1,p}(\R^3;\R^3)$ and in $\Hd{\R^3}$, see Proposition \ref{prop:density_full_space}.
    
    It remains to prove the same properties for $\d\in [0,1]^2$.
    To this end, we first select a sequence $(\d_k)_k\in (0,1]^2$ with $\d_k\to \d$ as $k\to \infty.$
    First consider $\phi \in C_c^{\infty}(\R^3;\R^3)$ again, and repeat similar calculations as in \eqref{Pcaldelta_L^p} to obtain
    \begin{align}\label{Pcald_difference}
        \norm{\Pcald \ffi - \Pcaldk\ffi}_{L^p(\R^3;\R^3)} \leq  C|\d-\d_k|\norm{\phi}_{W^{1,p}(\R^3;\R^3)} \to 0;
    \end{align}
    in particular, $\Pcald$ extends to a continuous map from $W^{1,p}(\R^3;\R^3)\to L^p(\R^3;\R^3).$
    Exploiting the fact that $\Pcaldk = \Qcaldk^{-1}$ for every $k\in\N$, we now prove that $\Qcald\Pcald\ffi = \ffi$ where we understand $\Qcald$ as a map from $L^p$ to $L^p$. 
    Indeed, it holds that
    $$\Qcald\Pcald\ffi - \ffi = \Qcald\Pcald\ffi - \Qcaldk \Pcaldk\ffi = (\Qcald-\Qcaldk)\Pcald\ffi + \Qcaldk(\Pcald\ffi-\Pcaldk \ffi)$$
    where the right-hand side vanishes as $k\to \infty$. 
    We treat the first term via Minkowski's integral inequality, and dominated convergence to show that
    \begin{align}
        \norm{(\Qcald-\Qcaldk)\Pcald\ffi}_{L^p(\R^3;\R^3)}\leq \int_{\R^3}Q_\rho(z)\norm{\Pcald\ffi - \Pcald\ffi(\cdot - (T_\d - T_{\d_k})z)}_{L^p(\R^3;\R^3)}\dd z \to 0,
    \end{align}
    while the second term can be handled via
    $$\norm{\Qcaldk(\Pcald\ffi-\Pcaldk \ffi)}_{L^p(\R^3;\R^3)} \leq \norm{\Pcald\ffi-\Pcaldk \ffi}_{L^p(\R^3;\R^3)} \to 0$$
    due to ($i$) and \eqref{Pcald_difference}.
    Hence, we obtain $\Dd \Pcald \ffi  = \nabla (\Qcald\Pcald \ffi) = \nabla \ffi$ so that \eqref{Pcaldelta_full_estimate} also holds for $\delta \in [0,1]^2$. 
    The map $\Pcald$ then extends to a bounded operator from $W^{1,p}(\R^3;\R^3)$ to $\Hd{\R^3}$ with operator norm bounded uniformly in $\d$ and $\Pcald$ is the right-inverse of $\Qcald$.
    That $\Pcald$ is also the left-inverse of $\Qcald$ can be similarly established via
    $$\Pcald \Qcald\ffi - \ffi = (\Pcald-\Pcaldk)\Qcald\ffi + \Pcaldk(\Qcald\ffi-\Qcaldk \ffi)$$
    on $C_c^\infty(\R^3;\R^3)$ and requires the density of $C_c^\infty(\R^3;\R^3)$ in $\Hd{\R^3}$ from Proposition \ref{prop:density_full_space}. \smallskip

    \textit{Step 3: Compact embedding.}  
    Finally, we prove the boundedness into $H^{\lambda,p}(\R^3;\R^3)$.
    For technical reasons, we will work with both Bessel and Gagliardo spaces as well as their respective norms and exploit that $W^{t,p}$ is continuously embedded in $H^{\lambda,p}$ for some suitable $t$.
    Given $\lambda_1 \in (\lambda,\sigma)$ and $\lambda_2 \in (1-\sigma,1)$ with $\lambda_1+\lambda_2=1$ we compute for any $h,h'\in \R^3$ that
    \begin{align*}
        \norm{\phi(\cdot+T_{\d}h+h')-\phi(\cdot+h')-&\phi(\cdot + T_{\d}h)+\phi}_{L^p(\R^3;\R^3)} \\
        &\leq C\abs{T_{\d}h}^{\lambda_2}\norm{\phi(\cdot+h')-\phi}_{H^{\lambda_2,p}(\R^3;\R^3)}\\
        &\leq C\abs{h}^{\lambda_2}\norm{\phi(\cdot+h')-\phi}_{L^p(\R^3;\R^3)}^{\lambda_1}\norm{\phi(\cdot+h')-\phi}^{\lambda_2}_{W^{1,p}(\R^3;\R^3)}\\
        &\leq C\abs{h}^{\lambda_2}\abs{h'}^{\lambda_1}\norm{\phi}_{W^{1,p}(\R^3;\R^3)}^{\lambda_1}\norm{\phi}_{W^{1,p}(\R^3;\R^3)}^{\lambda_2}\\
        &\leq C\abs{h}^{\lambda_2}\abs{h'}^{\lambda_1}\norm{\phi}_{W^{1,p}(\R^3;\R^3)},
    \end{align*}
    where the first and third inequality use a bound on translations in (fractional) Sobolev spaces, see~e.g.~\cite[Proposition~B.2]{BCC22}, and the second inequality is the interpolation inequality \eqref{eq:interpolation} for $s=\lambda_2$, cf.~\cite[Corollary~3.19]{BGS25}. We now obtain as before with Minkowski's integral inequality that
    \begin{align*}
        \norm{\Pcald \phi(\cdot+h')-\Pcald \phi}_{L^p(\R^3;\R^3)}\leq C\abs{h'}^{\lambda_1}\norm{\phi}_{W^{1,p}(\R^3;\R^3)}\int_{\R^3}\abs{h}^{1-\lambda_1}|\Div V_\rho(h)|\,dh,
    \end{align*}
    and the final integral is uniformly bounded in $\d$ just like before given that $\lambda_1 \in (\lambda,\sigma)$. 
    Hence, for any $t \in (\lambda, \lambda_1)$, we can estimate the Gagliardo semi-norm as follows
    \begin{align*}
        [\Pcald \phi]^p_{W^{t,p}(\R^3;\R^3)} &= \int_{B_1(0)} \frac{\norm{\Pcald \phi(\cdot + h') - \Pcald\phi}_{L^p(\R^3;\R^3)}^p}{\abs{h'}^{3+tp}} \dd h' + \int_{B_1(0)^c} \frac{C\norm{\Pcald\phi}^p_{L^p(\R^3;\R^3)}}{\abs{h'}^{3+tp}} \dd h'\\
        &\leq C\norm{\phi}^p_{W^{1,p}(\R^3;\R^3)} \Big(\int_{B_1(0)}\frac{1}{\abs{h'}^{3-(\lambda_1-t)p}}\dd h' + \int_{B_1(0)^c}\frac{1}{\abs{h'}^{3+tp}}\dd h'\Big) \leq C\norm{\phi}^p_{W^{1,p}(\R^3;\R^3)}.
    \end{align*}
    Therefore, we deduce that $\Pcald\colon W^{1,p}(\R^3;\R^3) \to W^{t,p}(\R^3;\R^3)$ is bounded with a uniform bound in $\d$, and together with the embedding $W^{t,p}(\R^3;\R^3) \hookrightarrow H^{\lambda,p}(\R^3;\R^3)$ as in \cite[Theorem 3.24]{BGS25} this finishes the proof.
    \medskip

    \textit{(iii)} This can be argued as in \cite[Theorem~4.1]{KrS24} by using Part (i) and (ii).
\end{proof}

The main reasons for developing the translation mechanism in the previous proposition are the following Poincar\'{e} inequality and compactness result, which hold uniformly in $\d$.
\begin{corollary}[Poincar\'e--Wirtinger inequality]\label{cor:poincare}
Let $O\subset \R^3$ be a bounded Lipschitz domain. There exists a constant $C>0$ independent of $\d$ such that
\[
\min_{h \in \Nd{O}}\norm{u-h}_{L^p(\R^3;\R^3)} \leq C \norm{\Dd u}_{L^p(O;\R^{3\times 3})} \quad \text{for all $u \in \Hd{O}$.}
\]
\end{corollary}
\begin{proof}
    This is an immediate consequence of Proposition~\ref{prop:translation}\,(iii), cf.~\cite[Lemma~4.7]{KrS24} for more details.
\end{proof}
\begin{corollary}[Uniform compactness]\label{cor:compactness}
    Let $O\subset \R^3$ be a bounded Lipschitz domain and $(\d_k)_k \subset[0,1]^2$ be a sequence of horizons with $\d_k\to \d\in[0,1]^2$ as $k\to \infty$. 
    If $u_k \in \Hdk{O}$ for $k\in\N$ satisfies
    \[
        \sup_k \norm{\Ddk u_k}_{L^p(O;\R^{3\times 3})}< \infty,
    \]
    then, up to a non-relabeled subsequence, there exist $h_k \in \Ndk{O}$ for $k\in\N$ such that $u_k-h_k \to u$ in $L^p(\R^3;\R^3)$ for some $u\in L^p(\R^3;\R^3)$.
\end{corollary}
\begin{proof}
    Define $v_k :=\Qcaldk u_k$ for $k\in\N$, then we find by Proposition~\ref{prop:translation}\,(i) that $(v_k)_k \subset W^{1,p}(O;\R^3)$ with
    \[
        \sup_k \norm{\nabla v_k}_{L^p(O;\R^{3\times 3})}=\sup_k \norm{\Ddk u_k}_{L^p(O;\R^{3\times 3})}< \infty.
    \]
    Hence, by subtracting mean values and extension, we find a bounded sequence $(\tilde{v}_k)_k \subset W^{1,p}(\R^3;\R^3)$ such that $\nabla \tilde{v}_k = \Ddk u_k$ on $O$ for all $k\in\N$. Using Proposition~\ref{prop:translation}\,(ii), we deduce that $\tilde{u}_k:=\Pcaldk \tilde{v}_k$ is a bounded sequence in $H^{\lambda,p}(\R^3;\R^3)$. Hence, up to a non-relabeled subsequence, we deduce that $\mathbbm{1}_{O_{\d_k}}\tilde{u}_k \to u$ in $L^p(\R^3;\R^3)$ {cf.~\cite[Theorem 1.1]{BCGS25}}. Moreover, it holds that $\Ddk \tilde{u}_k = \nabla \tilde{v}_k = \nabla v_k = \Ddk u_k$ on $O$. Hence, setting $h_k:=u_k-\mathbbm{1}_{O_{\d_k}}\tilde{u}_k \in \Ndk{O}$ yields the desired result.
\end{proof}
We now turn to the convergence of the nonlocal gradients with varying interaction radii.

\begin{lemma}\label{le:uniformconvergence}
   Let $(\d_k)_k \subset[0,1]^2$ with $\d_k\to \d\in[0,1]^2$ as $k\to \infty$. For any $\phi \in C_c^{\infty}(\R^3)$, it holds that
    \[
    \Qcaldk\phi \to \Qcald \phi \qand \Ddk \phi \to \Dd \phi \quad \text{uniformly as $k \to \infty$.}
    \]
\end{lemma}
\begin{proof}
    It is sufficient to prove the first statement in light of \eqref{Dzero}. For $\d_k=(\bar{\d}_k,\delta_{k,3}), \d=(\bar{\d},\d_3)\in [0,1]^2$ and any $x\in\R^3$, it holds that
    \begin{align*}
    \abs{\Qcaldk \phi(x) - \Qcald \phi(x)}&\leq \int_{\R^3} Q_\rho(z)\abs{\phi(x-T_{\d_k}z)-\phi (x-T_{\d}z)}\,dz \\
    &\leq \mathrm{Lip}(\phi) \int_{\R^3} Q_\rho(z) \abs{((\bar{\d}_k-\bar{\d})\bar{z},(\d_{k,3}-\d_{3})z_3)}\,dz\\
    &\leq C\mathrm{Lip}(\phi)\abs{\d_k-\d} \int_{\R^3}Q_\rho(z)\abs{z}\,dz \to 0 \quad \text{as $k\to \infty$}
    \end{align*}
    where $\mathrm{Lip}(\phi)$ denotes the Lipschitz constant of $\phi$.
\end{proof}

\begin{lemma}\label{le:convergenceP}
   Let $O\subset \R^3$ be an open set, $(\d_k)_k \subset[0,1]^2$ with $\d_k\to \d\in[0,1]^2$ as $k\to \infty$, and let $(v_k)_k \subset W^{1,p}(\R^3;\R^3)$ be a sequence weakly converging to $v \in W^{1,p}(\R^3;\R^3)$ such that $\supp v_k \subset B_R(0)$ for some $R>0$. Then, it holds that
    \[
    \mathbbm{1}_{O_{\d_k}}\Pcaldk v_k \to \mathbbm{1}_{O_{\d}}\Pcald v \quad \text{in $L^p(\R^3;\R^3)$ as $k\to\infty$.}
    \]
\end{lemma}
\begin{proof}
    First consider $\phi \in C_c^{\infty}(\R^3;\R^3)$, then we can use the Lipschitz continuity of $\phi$ to compute for almost every $x\in\R^3$ that
    \begin{align*}
        \abs{\Pcaldk \phi (x) - \Pcald \phi(x)} &\leq \int_{\R^3} \abs{\phi(x+T_{\d_k}h) - \phi(x+T_{\d}h)}\abs{\Div V_\rho(h)}\,dh\\
        &\leq C \abs{\d_k-\d}\int_{\R^3} \abs{h}\abs{\Div V_\rho(h)}\,dh \to 0,
    \end{align*}
    as $k\to\infty$. This shows that $\Pcaldk \phi \to \Pcald \phi$ uniformly in $\R^3$. For the general case, we note that $(\Pcaldk v_k)_k$ is a bounded sequence in $H^{\lambda,p}(\R^3;\R^3)$ by Proposition~\ref{prop:translation}\,(ii). Hence, up to a non-relabeled subsequence, we deduce that $\Pcaldk v_k \to u$ in $L^p_{\rm loc}(\R^3;\R^3)$. Now, we can compute for all $\phi \in C_c^{\infty}(\R^3)$ that
    \begin{align*}
        \int_{\R^3} u  \phi \,dx &= \lim_{k \to \infty} \int_{\R^3}\Pcaldk v_k \phi\,dx = \lim_{k \to \infty} \int_{\R^3} v_k \Pcaldk \phi\,dx \\
        &= \int_{\R^3}v \Pcald \phi\,dx = \int_{\R^3} \Pcald v \phi\,dx,
    \end{align*}
    where we have used Fubini's theorem to move $\Pcaldk$ and $\Pcald$ from one function to the other, and the fact that $v_k \to v$ in $L^p(\R^3;\R^3)$ with $\supp v_k \subset B_R(0)$ and the uniform convergence $\Pcaldk \phi \to \Pcald \phi$. We conclude that $u=\Pcald v$, and since the limit is independent of the chosen subsequence, the convergence $\mathbbm{1}_{O_{\d_k}}\Pcaldk v_k \to \mathbbm{1}_{O_{\d}}\Pcald v$ in $L^p(\R^3;\R^3)$ also holds for the full sequence.
\end{proof}

Since in the $\Gamma$-limit of the thin-film analysis later on we find the constraint $(\Dd u)e_3 =0$ in $\Omega$, we want to show that this implies that $u$ is constant in the third variable up to an equivalent representative.

\begin{lemma}\label{le:constant_x3}
    Let $\Omega=\omega\times I$ with $\omega\subset\R^2$ a bounded Lipschitz domain, $I\subset \R$ an open interval, and let $\d\in[0,1]^2$. If $u \in \Hd{\O}$ satisfies $(\Dd u)e_3 =0$ a.e.~in $\O$, then there exists an $h \in \Nd{\O}$ (see \eqref{cursed_functions}) such that $\partial_3 (u-h)=0$ in $\O_\d$ in a distributional sense.
\end{lemma}
\begin{proof}
    Setting $v:=\Qcald u \in W^{1,p}(\O;\R^3)$, we get $\nabla v = \Dd u$ on $\O$, see Proposition \ref{prop:translation}~(i). 
    In particular, $\partial_3 v=0$, so $v$ is constant in the third variable. Hence, by extending to alll of $\R^2$, there exists a $\bar{v} \in W^{1,p}(\R^2;\R^3)$ with $v(x)=\bar{v}(\bar{x})$ for a.e.~$x \in \O$. 
    Now, we can define the operator 
    \begin{align}\label{Ptwo}
        \Pcaltwod\colon W^{1,p}(\R^2;\R^3) \to H^{\bar{\rho},p}_{\bar{\d}}(\R^2;\R^3)
    \end{align}
    associated to the kernel $\bar{\rho}$ from Lemma~\ref{le:2dgradient} as in Proposition~\ref{prop:translation}\,(ii) since $\bar{\rho}$ satisfies all the same assumptions as $\rho$, cf.~Remark~\ref{rem:rhobaradmissible}. 
    If $\bar{\d}=0$, the operator $\Pcaltwod$ is simply the identity and we interpret $H^{\bar{\rho},p}_{\bar{\d}}(\R^2;\R^3)=W^{1,p}(\R^2;\R^3)$. On the other hand, the case $\bar\d>0$ leads back to the well-known two-dimensional nonlocal theory with radial symmetry of the kernel.
    We now set $\bar{u}:=\Pcaltwod \bar{v}$, which satisfies $\Qcaltwod \bar{u}=\bar{v}$, where $\Qcaltwod$ is defined analogously to $\Qcald$ as its two-dimensional in-plane counterpart with kernel $\bar{\rho}$. The function $\tilde{u}$ defined as $\tilde{u}(x):=\mathbbm{1}_{\Omega_{\d}}(x)\bar{u}(\bar{x})$ for a.e.~$x \in \R^3$ then satisfies
    \begin{align*}
        \Dd \tilde{u}(x)=\nabla \Qcald \tilde{u}(x) = (\bar{\nabla} \Qcaltwod\bar{u}(\bar{x}),0)=(\bar{\nabla}\bar{v}(\bar{x}),0) = \nabla v(x) = \Dd u(x),
    \end{align*}
    for a.e.~$x \in \Omega$. Hence, setting $h:=u-\tilde{u} \in \Nd{\O}$ finishes the proof, given that $\partial_3 \tilde{u}=0$ in $\Omega_{\d}$.
\end{proof}

Finally, to make sure that we can strongly approximate the specific representative which is constant in the last variable (see Lemma \ref{le:constant_x3}), we need an approximation result for the spaces $\Ndk{\O}$ and $\Nd{\O}$, which, in turn, relies on the density result in Proposition \ref{prop:density}.

\begin{lemma}\label{le:Napproximation}
     Let $\Omega=\omega\times I$ with $\omega\subset \R^2$ a bounded Lipschitz domain, $I\subset \R$ an open interval, and let $(\d_k)_k\subset [0,1]^2$ with $\d_k \to \d\in [0,1]^2$ as $k\to \infty$. For every $h \in \Nd{\Omega}$, there exists $h_k \in \Ndk{\Omega}$ for every $k\in \N$ such that $h_k \to h$ in $L^p(\R^3;\R^3)$.
\end{lemma}
\begin{proof}
    First we show that there exist $\tilde{h}_k \in \Hdk{\Omega}$ for every $k\in\N$ such that $\tilde{h}_k \to h$ in $L^p(\R^3;\R^3)$ and $\Ddk \tilde{h}_k \to 0$ in $L^p(\Omega;\R^{3\times3})$ as $k \to \infty$. 
    To this aim, we use the density result of Proposition~\ref{prop:density}, to find a sequence $(\phi_j)_j \subset C_c^{\infty}(\R^3;\R^3)$ such that $\phi_j \to h$ in $L^p(\Omega_{\d};\R^3)$ and $\Dd \phi_j \to 0$ in $L^p(\Omega;\R^{3\times 3})$. 
    In particular, using Lemma~\ref{le:uniformconvergence}, we deduce that
    \[
    \lim_{j \to \infty} \lim_{k \to \infty}\norm{\Ddk \phi_j}_{L^p(\Omega;\R^{3\times3})}=\lim_{j \to \infty} \norm{\Dd \phi_j}_{L^p(\Omega;\R^{3\times3})}=0.
    \]
    Hence, we can find a suitable diagonal sequence such that $\tilde{h}_k := \mathbbm{1}_{\Omega_{\d_k}} \phi_{j_k}$ satisfies $\tilde{h}_k \to h$ in $L^p(\R^3;\R^3)$ and $\Ddk \tilde{h}_k \to 0$ in $L^p(\Omega;\R^{3\times3})$ as $k \to \infty$. \smallskip

    We now define our actual desired sequence as
    \[
    h_k := \argmin_{h \in \Ndk{\Omega}} \norm{\tilde{h}_k - h}_{L^p(\R^3;\R^3)} \quad \text{for $k\in\N$;}
    \]
    note that the minimization problem has a unique minimizer due to the strict convexity of the $L^p$-norm and the fact that $\Ndk{\Omega}$ is weakly closed in $L^p(\R^3;\R^3)$. 
    It remains to prove that $h_k \to h$ in $L^p(\R^3;\R^3)$. 
    This follows from the Poincar\'{e}--Wirtinger inequality in Corollary~\ref{cor:poincare} because
    \[
    \norm{h_k - \tilde{h}_k}_{L^p(\R^3;\R^3)}\leq C \norm{\Ddk \tilde{h}_k}_{L^p(\Omega;\R^{3\times3})} \to 0 \quad \text{as $k \to \infty$,}
    \]
    where we used that the constant is uniform in $k$. 
    Hence, the sequence $(h_k)_k$ has the same limit as $(\tilde{h}_k)_k$ in $L^p(\R^3;\R^3)$, which is $h$.
\end{proof}

\begin{remark}[Anisotropic nonlocal gradients in $n$ dimensions]\label{rem:higher_dimensions}
    The definition of our nonlocal gradients, whose interaction neighborhoods are ellipsoids with a (possibly vanishing) in-plane and out-of-plane principal radius, is primarily motivated by the analysis of thin films in nonlocal hyperelasticity, see Section \ref{sec:thin_films} below.
    However, this theory can be easily generalized to arbitrary dimensions $n\in\N$ without coupling any directions, giving rise to $n$ (different) principal radii.\smallskip

    More precisely, we define for $\d = (\d_1,\ldots,\d_n)\in[0,1]^n$ and for $\rho$ satisfying (H0)--(H4) the anisotropic nonlocal gradient $\Dd$ analogously to Lemma \ref{lem:anisotropic_nonlocal_gradients} and Definition \ref{def:weak_limit_gradients}, where $T_\d$ is now the diagonal matrix with entries $\d_1,\ldots,\d_n\in[0,1]$.
    In particular, $\Dd$ can again be represented in terms of a convolution if $\d\in(0,1]^n$ as in Remark \ref{rem:positive_horizon}.
    
    Also the Leibniz rule in Lemma \ref{lem:leibniz} and Lemma \ref{lem:product} have natural generalizations to $n$ dimensions with $n$ principal radii; the distinctions about vanishing components of $\d$ simply become more cumbersome.
    While the density of $C_c^\infty$-functions in $H^{\rho,p}_\d$ on the full space $\R^n$ (see Proposition \ref{prop:density_full_space}) is straightforward, we need to be more careful when working with bounded Lipschitz domains $\Omega\subset\R^n$.
    If $\d\in[0,1]^n$, then our proof strategy in Proposition \ref{prop:density} requires $\Omega$ to be an $n$-dimensional cuboid. 
    We would like to point out that this restriction is completely absent if no component of $\d$ vanishes, as this product structure is merely a technical requirement in light of the different cases in Lemma \ref{lem:product}.

    The translation mechanism in Proposition \ref{prop:translation} can be generalized canonically without any additional effort. Hence, the follow-up results in Corollary \ref{cor:poincare}--Proposition \ref{le:convergenceP} as well as Lemma \ref{le:Napproximation} also hold in $n$ dimensions.

    The results in Lemma \ref{le:2dgradient} and Lemma \ref{le:constant_x3} are more specific to our thin-film analysis later on but can be generalized as well. Particularly, the strategy in the proof of Lemma \ref{le:2dgradient} can be iterated since integrating $\rho$ in one variable preserves the assumptions (H0)--(H4).
    If one is not interested in consistency with lower-dimensional gradients, then assumption (H1) can even be weakened again to its source version in \cite{BMS24,CKS25}.
\end{remark}

\section{Thin films in nonlocal hyperelasticity}\label{sec:thin_films}

For the hyperelastic nonlocal thin membrane with reference configuration $\Omega^\e=\omega\times (0,\e)$, where $\omega\subset\R^2$ is a bounded Lipschitz domain and has height $\e>0$, we consider two (in-plane and out-of-plane) interaction radii $\d = \d(\eps)\in[0,1]^2$. 
We thus define its stored elastic energy per unit volume $\Ecal_\eps\colon L^p(\R^3;\R^3)\to[0,\infty]$ as
\begin{align}\label{thin_state_energy}
    \Ecal_\e(v)= 
    \begin{cases}
        \displaystyle\frac{1}{\e}\int_{\Omega^\e} W(D_{\rho,\d} v)\dd y &\text{ if } v\in H^{\rho,p}_\delta(\O^\e;\R^3),\\
        \infty  & \text{ otherwise,}
    \end{cases}
\end{align}
with $\Dd$ and $\Hd{\O^\eps}$ as in Definition \ref{def:weak_limit_gradients}. 
The energy density $W\colon\R^{3\times 3}\to [0,\infty)$ shall be continuous and satisfy the growth condition
\begin{align}\label{Wgrowth}
	c|F|^p - C \leq W(F)\leq C(|F|^p + 1)\quad \text{ for all } F\in \R^{3\times 3}
\end{align}
for some constants $C,c>0$.
After the typical thin-film rescaling argument
\begin{align*}
	u(x) = v(y)\quad\text{with}\quad x= T_{(1,\eps)}^{-1}y,
\end{align*}
the energy \eqref{thin_state_energy} turns into $\Ical_\eps\colon L^p(\R^3;\R^3)\to[0,\infty]$, given by
\begin{align}\label{rescaled_state_energy}
   \Ical_\eps(u)= 
    \begin{cases}
        \displaystyle\int_{\Omega} W\big((D_{\rho,\d^\e} u) T_{(1,\e)}^{-1}\big)\dd x &\text{ if } u\in \He ,\\
        \infty  & \text{ otherwise}
    \end{cases}
\end{align}
with $\Omega:=\Omega_1$ and $\d^\e \coloneqq (\bar{\d}, \frac{\d_3}{\e})\in[0,1]^2$, where we implicitly assume that $\d_3\leq \e$.
Note that the choice $\d = (0,0)$ leads to the classical hyperelastic energies \cite{LeR95} aside from boundary conditions (or \cite{FJM02, FJM06} if scaled accordingly), so that our approach is a natural generalization of the local theory to the nonlocal setting.

\begin{theorem}\label{th:state}
    Let $p\in(1,\infty)$ and $\d=\d(\e)\in [0,1]^2$ for $\e>0$ be given and assume that $\d^\e:=(\bar\d,\frac{\d_3}{\e})\to \d^0\in [0,1]^2$ as $\e\to 0$. 
    Then, the following holds:
    \begin{itemize}
        \item[(i)] The sequence $(\Ical_\e)_\e$ in \eqref{rescaled_state_energy} $\Gamma$-converges with respect to the strong topology in $L^p(\R^3;\R^3)$ to the functional $\Ical\colon L^p(\R^3;\R^3) \to [0,\infty]$ defined by
    \begin{align}\label{limit_state_energy}
        \Ical(u)=
        \begin{cases}
            \displaystyle \int_\Omega \Wmin^\qc(\Dbar u)\dd x & \text{ if }u\in \Hzero, (\Dzero u)\restrict{\Omega}e_3 =0,\\
            \infty  &\text{ otherwise,}
        \end{cases}
    \end{align}
    where $\Dbar$ are the first two columns of $\Dzero$ as in Definition \ref{def:weak_limit_gradients}, and $\Wmin^\qc$ denotes the quasiconvex envelope of the function
    \[
    \Wmin\colon \R^{3\times 2} \to [0,\infty), \quad \Wmin(\bar{A})=\min_{a \in \R^3} W(\bar{A}\,|\, a).
    \]
    
    \item[(ii)] For every sequence $(u_\e)_\e\subset L^p(\R^3;\R^3)$ with $\sup_\e \Ical_\e(u_\e) < \infty$, there is a subsequence (not relabeled) such that $(u_\e)_\e$ converges in $L^p(\R^3;\R^3)$ -- up to representatives in $\Ne$ -- to some $u\in \Hzero$ with $\partial_3 u =0$ on $\Omega_{\d^0}$, i.e., there exist $h_\e\in\Ne$ such that $u_\e - h_\e\to u$ in $L^p(\R^3;\R^3)$ and $\De u_\e \weakly \Dzero u$ in $L^p(\Omega;\R^3)$.
    \end{itemize}
\end{theorem}

\begin{remark}[Lower-dimensional model]\label{rem:2drep}
    The functional in \eqref{limit_state_energy} can be reformulated to highlight its two-dimensional character.
    If $u\in \Hzero$ with $(\Dzero u)e_3 =0$ on $\Omega$, then $v\coloneqq\Qcalzero u\in W^{1,p}(\Omega;\R^3)$ satisfies $\partial_3 v =0$. 
    In particular, the first two columns $\bar\nabla v = (\partial_1 v\,|\, \partial_2 v)$ of $\nabla v$ are independent of $x_3$. 
    Since $\nabla v = \Dzero u$, the first two columns $\Dbar u$ of $\Dzero u$ are also independent of $x_3$.
    
    Moreover, there exists an $h\in \Nzero$ such that $\bar{u}\coloneqq u-h$ is independent of $x_3$, see Lemma \ref{le:constant_x3}, and can thus be interpreted as a function on $\omega_{\bar{\d}^0}$.
    These two observations lead to
    \begin{align*}
        \Ical(u) = \int_{\omega} \Wmin^\qc( \Dtwo \bar{u}) \dd \bar{x},
    \end{align*}
    where $\Dtwo$ is the two-dimensional (non)local gradient defined in Lemma~\ref{le:2dgradient}, see also Remark \ref{rem:rhobaradmissible}.
\end{remark}

\begin{proof}[Proof of Theorem~\ref{th:state}]
\textit{Step 1: Compactness.} Let $(u_\epsilon)_\epsilon \subset L^p(\R^3;\R^3)$ be a sequence such that
\[
\sup_{\epsilon}\Ical_\epsilon(u_\epsilon) < \infty.
\]
Then, due to the lower bound \eqref{Wgrowth} on $W$, we find that $u_\epsilon \in \He$ for all $\epsilon >0$ and
\[
\sup_{\epsilon} \norm{(\De u_\epsilon)}_{L^p(\Omega;\R^{3\times3})}\leq \sup_{\epsilon} \norm{(\De u_\epsilon) T_{(1,\epsilon)}^{-1}}_{L^p(\Omega;\R^{3\times3})} < \infty.
\]
We conclude from Corollary~\ref{cor:compactness} that there are (up to a non-relabeled subsequence) $\tilde{h}_\epsilon \in \Ne$ such that $u_\epsilon-\tilde{h}_\epsilon \to \tilde{u}$ in $L^p(\R^3;\R^3)$. To show that $\tilde{u} \in \Hzero$, we first note that $\tilde{u}=0$ in $(\Omega_{\d^0})^c$ since $\tilde{u}_\epsilon=0$ on $(\Omega_{\d^\e})^c$ for all $\epsilon >0$ and $\d^\e \to \d^0$.

Moreover, we find $\De u_\epsilon \rightharpoonup U$ in $L^p(\Omega;\R^{3\times3})$ with $Ue_3=0$. Via integration by parts we then find for all $\psi \in C_c^{\infty}(\Omega;\R^{3})$ that
\begin{align*}
    \int_{\Omega} U \cdot \psi \,dx &= \lim_{\epsilon \to 0} \int_{\Omega} \De u_\epsilon \cdot \psi\,dx = \lim_{\epsilon \to 0} \int_{\Omega} \De(u_\epsilon-\tilde{h}_\e) \cdot \psi\,dx \\
    &= -\lim_{\epsilon \to 0}\int_{\R^3}(u_\epsilon -\tilde{h}_\e)  \dive \psi \,dx = -\int_{\R^3}\tilde{u} \divzero \psi\,dx,
\end{align*}
where we have used Lemma~\ref{le:uniformconvergence} in the last equality together with the convergence $u_\epsilon -\tilde{h}_\e \to \tilde{u}$ in $L^p(\R^3;\R^3)$ and the fact that $u_\epsilon=\tilde{h}_\e = 0$ outside $\Omega_{\d^\e}$ for all $\epsilon >0$. We conclude that $\tilde{u} \in \Hzero$ with $\Dzero \tilde{u} = U$, and, in particular, $(\Dzero \tilde{u})e_3 =Ue_3=0$.
By Lemma~\ref{le:constant_x3}, we now find a $h \in \Nzero$ such that $u:=\tilde{u}-h$ is constant in $x_3$ on $\Omega_{\d^0}$. Using Lemma~\ref{le:Napproximation}, we can find $h_\epsilon \in \Ne$ such that $u_\epsilon - h_\epsilon \to u$ in $L^p(\R^3;\R^3)$, as desired. \smallskip

\textit{Step 2: Liminf-inequality.} Given the previous compactness step, we now define $v_\epsilon :=\Qcale (u_\epsilon-h_\e)$ for $\epsilon >0$ and $v:=\Qcalzero u$ with $\Qcale$ and $\Qcalzero$ as in Proposition \ref{prop:translation}. Then, it holds that $\nabla v_\epsilon = \De u_\epsilon$ and $\nabla v = \Dzero u$. We then use the bound in Proposition \ref{prop:translation} (i) and Lemma \ref{le:uniformconvergence} to infer that $v_\epsilon \rightharpoonup v$ in $W^{1,p}(\Omega;\R^3)$. In the spirit of the classical dimension reduction result from \cite[Theorem~2]{LeR95}, we now deduce that
\begin{align*}
    \liminf_{\epsilon \to 0}\Ical_\epsilon(u_\epsilon)
    =\liminf_{\epsilon \to 0} \int_{\Omega}W((\nabla v_\epsilon) T_{(1,\epsilon)}^{-1})\,dx
    \geq \int_{\Omega} \Wmin^\qc(\bar{\nabla} v)\,dx=\Ical(u).
\end{align*}\smallskip

\textit{Step 3: Recovery sequence.} Let $u \in \Hzero$ with $(\Dzero u)e_3=0$ in $\Omega$, then we can define $v:=\Qcalzero u \in W^{1,p}(\Omega;\R^3)$ such that $\partial_3 v =0$. Using the classical recovery sequence from \cite[Theorem~2]{LeR95} (without boundary conditions), we find a bounded sequence $(v_\epsilon)_\epsilon \subset W^{1,p}(\Omega;\R^3)$ such that $v_\epsilon \to v$ in $L^p(\Omega;\R^3)$ and
\begin{equation}\label{eq:recoveryclassical}
    \lim_{\epsilon \to 0} \int_{\Omega} W( (\nabla v_\epsilon) T_{(1,\epsilon)}^{-1})\,dx = \int_{\Omega} \Wmin^\qc (\bar{\nabla}v)\,dx.
\end{equation}
By extending the sequence $(v_\epsilon)_\epsilon$ and $v$ outside $\Omega$, we obtain a bounded sequence $(\tilde{v}_\epsilon)_\epsilon \subset W^{1,p}(\R^3;\R^3)$ and $\tilde{v} \in W^{1,p}(\R^3;\R^3)$ such that $\tilde{v}_\epsilon \to \tilde{v}$ in $L^p(\R^3;\R^3)$. In fact, we may assume $\supp \tilde{v}_\epsilon \subset B_R(0)$ for some $R>0$. Then, by Lemma~\ref{le:convergenceP}, we infer that 
\[
\tilde{u}_\epsilon:=\mathbbm{1}_{\Omega_{\d^\e}}\Pcale \tilde{v}_\epsilon \to \mathbbm{1}_{\Omega_{\d^0}}\Pcalzero \tilde{v}=:\tilde{u} \quad \text{in $L^p(\R^3;\R^3)$.}
\]
In addition, $\Dzero \tilde{u}=\nabla v = \Dzero u$ on $\Omega$, so $h:=\tilde{u}-u \in \Nzero$. Using Lemma~\ref{le:Napproximation}, we find a sequence $h_\epsilon \in \Ne$ for $\epsilon >0$ such that $h_\epsilon \to h$ in $L^p(\R^3;\R^3)$. If we now set $u_\epsilon:=\tilde{u}_\epsilon-h_\epsilon,$ we infer that
\[
u_\epsilon \to \tilde{u}-h=u \quad \text{in $L^p(\R^3;\R^3)$.}
\]
Finally, we now find
\begin{align*}
    \lim_{\epsilon \to 0} \Ical_\epsilon(u_\epsilon) = \lim_{\epsilon \to 0} \int_{\Omega} W((\nabla v_\epsilon) T_{(1,\epsilon)}^{-1})\,dx = \int_{\Omega} \Wmin^\qc (\bar{\nabla}v)\,dx = \Ical(u),
\end{align*}
by using \eqref{eq:recoveryclassical} and the fact that $\De u_\epsilon =\De \tilde{u}_\epsilon = \nabla v_\epsilon$ in $\Omega$ for all $\epsilon >0$ given that $h_\epsilon \in \Ne$.
\end{proof}

As an application of the $\Gamma$-convergence result, we present the convergence of minimizers. In this case, we can also include applied forces, but they have to be orthogonal to the functions with zero nonlocal gradient, since otherwise the energy would become unbounded from below. Since $\Ne$ depends on~$\epsilon$, the force densities must also depend on $\e$. We can choose such a sequence by fixing $V \in L^{p'}(\R^3;\R^{3 \times 3})$ with $\frac{1}{p}+ \frac{1}{p'}=1$ such that $V=0$ in $\Omega^c$ for which
\[
\dive V \to \divzero V \in L^{p'}(\R^3;\R^3);
\]
in particular, all these operators need to exist in a weak sense. Then, we can define our sequence of densities as 
\begin{equation}\label{eq:forces}
f_\epsilon :=\dive V \qand f_0:=\divzero V,
\end{equation}
for which we find
\[
\int_{\R^3} f_\epsilon \cdot h\,dx = \int_{\R^3} \dive V \cdot h\,dx = -\int_{\Omega} V \cdot \De h\,dx = 0 \quad \text{for all $h \in \Ne$}
\]
and the same for $f_0$ and $\Nzero$; here, we used the nonlocal integration by parts for two non-smooth functions, whose validity can be shown by approximating $h$ by smooth functions, cf.~Proposition~\ref{prop:density}.
 
\begin{corollary}\label{cor:minimizers}
    Let $p\in(1,\infty)$ and $\d=\d(\e)\in [0,1]^2$ for $\e>0$ be given and assume that $\d^\e=(\bar\d,\frac{\d_3}{\e})\to \d^0\in [0,1]^2$ as $\e\to 0$.
    Moreover, let $(f_\epsilon)_\e,f_0$ be as in \eqref{eq:forces}. Then, any sequence of minimizers of
    \[
    u \mapsto \int_{\Omega} W\big((D_{\rho,\d^\e} u) T_{(1,\e)}^{-1}\big)\dd x - \int_{\Omega_{\d^\e}}f_\epsilon \cdot u \,dx \quad \text{for $u \in \He$}
    \]
    converges in $L^p(\R^3;\R^3)$ up to subsequence and representatives in $\Ne$ to a minimizer of
    \[
    \bar{u} \mapsto \int_{\omega}\Wmin^\qc(\Dtwo \bar{u})\,d\bar{x} - \int_{\omega_{\bar{\d}^0}} \bar{f}_0 \cdot \bar{u}\,dx
    \]
    over all $\bar{u} \in H^{\bar{\rho},p}_{\bar{\d}^0}(\omega;\R^3)$ (cf.~\eqref{Htwod}) with $\bar{f}_0(\bar{x}):=\int_{\R}f_0(\bar{x},x_3)\,dx_3.$
\end{corollary}
\begin{proof}
    This is an immediate consequence of Theorem~\ref{th:state} by using the properties of $\Gamma$-convergence together with the convergence of the force contributions and the fact that all functionals are invariant under translations in $\Ne$. Additionally, in the limit we used the two-dimensional representation of the $\Gamma$-limit from Remark~\ref{rem:2drep}.
\end{proof}
While this is a mathematically satisfactory result, from a practical perspective, having to deal with the functions in $\Ne$ is quite unnatural; not least for the fact that this implies that there are non-constant deformations which do not incur additional stored energy. Therefore, we will now also pursue an approach where we include a penalty or regularization term in the spirit of \cite[Section 6]{Sil17}, which accounts for the functions in $\Ne$. This should stabilize the functional so that the compactness and convergence of minimizers holds without taking different representatives. 

We consider the stabilizing term
\[
v \mapsto \frac{1}{\epsilon}\int_{\Omega^\e_\d} \abs{v}^p\,dy + \frac{1}{\epsilon^2}\int_{\Omega^\e_\d}\int_{\Omega^\e_\d}\frac{\abs{v(\bar{y},y_3)-v(\bar{y},y_3')}^p}{\abs{y_3-y_3'}}\,dy\,dy',
\]
which will be added to the energy $\Ecal_\epsilon$; the first term ensures that large deformations with zero nonlocal gradient are penalized, while the second nonlocal term is there to make sure that the limiting functions are constant in the third argument. After rescaling, the stabilizing term takes the form
\[
\Scal_\epsilon (u):= \int_{\Omega_{\d^\e}}\abs{u}^p\,dx + \int_{\Omega_{\d^\e}}\int_{\Omega_{\d^\e}}\frac{\abs{u(\bar{x},x_3)-u(\bar{x},x_3')}^p}{\abs{\epsilon(x_3-x_3')}}\,dx\,dx' \quad \text{for $u \in L^p(\R^3;\R^3)$.}
\]
Hence, we are interested in the $\Gamma$-convergence of the functionals 
\begin{equation}\label{eq:stabilizedI}
\Ical_\epsilon^{\rm stab}:=\Ical_\epsilon+\lambda\Scal_\epsilon
\end{equation}
with $\Ical_\e$ as in \eqref{rescaled_state_energy} and for some stabilization parameter $\lambda>0$. We have the following result.

\begin{theorem}\label{th:stabilized}
    Let $p\in(1,\infty)$ and $\d=\d(\e)\in [0,1]^2$ for $\e>0$ be given and assume that $\d^\e:=(\bar\d,\frac{\d_3}{\e})\to \d^0\in [0,1]^2$ as $\e\to 0$.
    Then, the following holds:
    \begin{itemize}
        \item[(i)] The sequence $(\Ical_\e^{\rm stab})_\e$ in \eqref{eq:stabilizedI} $\Gamma$-converges with respect to the weak topology in $L^p(\R^3;\R^3)$ to the functional $\Ical^{\rm stab}\colon L^p(\R^3;\R^3) \to [0,\infty]$ defined by
        \begin{align*}
            \Ical^{\rm stab}(u)=
            \begin{cases}
                \displaystyle\int_\Omega \Wmin^\qc(\Dbar u)\dd x +\lambda\int_{\Omega_{\d^0}}\abs{u}^p\,dx  & \text{ if $u\in \Hzero$ with $\partial_3 u|_{\Omega_{\d^0}} =0$},\\
                \infty  &\text{ otherwise.}
            \end{cases}
        \end{align*}
        In particular, if $\bar{u} \in H^{\bar{\rho},p}_{\bar{\d}^0}(\omega;\R^3)$ (cf.~\eqref{Htwod}) is the two-dimensional representation of $u$, we have
        \[
        \Ical^{\rm stab}(u) = \int_{\omega}\Wmin^\qc(\Dtwo \bar{u})\,d\bar{x} + \lambda \int_{\omega_{\bar{\d}^0}}d_{\d^0}\abs{\bar{u}}^p\,d\bar{x}
        \]
        with
        $$d_{\d^0}\colon \omega_{\bar{\d}^0}\to [0,\infty),\ \bar{x}\mapsto \begin{cases}1+2\d_3^0\sqrt{1-(\frac{1}{\bar{\d}^0}\dist(\bar{x},\omega))^2} & \text{ if }\bar{\d}^0 \neq 0,\\ 1+2\d_3^0 & \text{ if } \bar{\d}^0=0.\end{cases}$$
    \item[(ii)] For every sequence $(u_\e)_\e\subset L^p(\R^3;\R^3)$ with $\sup_\e \Ical_\e(u_\e) < \infty$, there is a subsequence (not relabeled) such that $(u_\e)_\e$ converges weakly in $L^p(\R^3;\R^3)$ to some $u\in \Hzero$ with $\partial_3 u =0$ on $\Omega_{\d^0}$.
    \end{itemize}
\end{theorem}
\begin{proof}
    \textit{Step 1: Compactness.} The compactness with respect to the weak topology is clear due to the term
    \[
    u \mapsto \lambda\int_{\Omega_{\d^\e}}\abs{u}^p\,dx = \lambda\int_{\R^3}|u|^p\dd x
    \]
    present in the stabilization functional. \smallskip

    \textit{Step 2: Liminf-inequality.} For the liminf-inequality, let $u \in L^p(\R^3;\R^3)$ and $(u_\epsilon)_\epsilon \subset L^p(\R^3;\R^3)$ be a sequence converging weakly to $u$. We further suppose without loss of generality that 
    \[
    \sup_{\epsilon} \Ical_\epsilon^{\rm stab}(u_\epsilon) < \infty.
    \]
    It is not hard to see that one can adapt the proof of the liminf-inequality of Theorem~\ref{th:state}\,(i) to the case of weak convergence; that is, we have
    \begin{equation}\label{eq:ieliminf}
    \liminf_{\epsilon \to 0} \Ical_\epsilon(u_\epsilon) \geq \Ical(u)=\int_\Omega \Wmin^\qc(\Dbar u)\dd x.
    \end{equation}
    Furthermore, the weak lower semicontinuity of the $L^p$-norm yields
    \[
    \liminf_{\epsilon \to 0} \int_{\Omega_{\d^\e}}\abs{u_\e}^p\,dx \geq \liminf_{\epsilon \to 0} \int_{\Omega'}\abs{u_\e}^p\,dx \geq   \int_{\Omega'}\abs{u}^p\,dx
    \]
    for any set $\Omega' \subset \Omega_{\d^\e}$ for all $\epsilon$ large enough. Hence, letting $\Omega'$ tend to $\Omega_{\d^0}$ with the monotone convergence theorem, we find
    \begin{equation}\label{eq:lpliminf}
    \liminf_{\epsilon \to 0} \int_{\Omega_{\d^\e}}\abs{u_\e}^p\,dx \geq  \int_{\Omega_{\d^0}}\abs{u}^p\,dx.
    \end{equation}
    In a similar manner, we find using the weak lower semicontinuity of convex integral functionals, that
    \begin{align*}
    \liminf_{\e \to 0} \int_{\Omega_{\d^\e}}\int_{\Omega_{\d^\e}}\frac{\abs{u_\epsilon(\bar{x},x_3)-u_\epsilon(\bar{x},x_3')}^p}{\abs{\epsilon(x_3-x_3')}}\,dx\,dx' &\geq \liminf_{\e \to 0} \frac{1}{\eta}\int_{\Omega'}\int_{\Omega'}\frac{\abs{u_\e(\bar{x},x_3)-u_\e(\bar{x},x_3')}^p}{\abs{x_3-x_3'}}\,dx\,dx' \\
    &\geq \frac{1}{\eta}\int_{\Omega'}\int_{\Omega'}\frac{\abs{u(\bar{x},x_3)-u(\bar{x},x_3')}^p}{\abs{x_3-x_3'}}\,dx\,dx',
    \end{align*}
    where $\eta>0$ is arbitrary and $\Omega '\subset \Omega_{\d^\e}$ for all large $\e$. Letting first $\Omega'$ tend to $\Omega_{\d^0}$ and subsequently $\eta \to 0$, we infer that
    \[
    \int_{\Omega_{\d^0}}\int_{\Omega_{\d^0}}\frac{\abs{u(\bar{x},x_3)-u(\bar{x},x_3')}^p}{\abs{x_3-x_3'}}\,dx\,dx'=0.
    \]
    Therefore, we find that $u(\bar{x},x_3)=u(\bar{x},x_3')$ for a.e.~$x,x' \in \Omega_{\d^0}$. Hence, we deduce that $\partial_3 u=0$ on $\Omega_{\d^0}$ in a distributional sense. Adding \eqref{eq:ieliminf} and \eqref{eq:lpliminf} together, now yields the liminf-inequality. \smallskip

    \textit{Step 3: Recovery-sequence.} We split this step into two parts. We first recover the energy without the quasiconvexification of $\Wmin$ and deal with approximating the quasiconvexification via relaxation separately.\medskip
		
	\textit{Step 3a: Recovery without relaxation.} Let $u\in \Hzero$ with $\Ical^{\rm stab}(u)<\infty$, and denote by $\bar{u} \in H^{\bar{\rho},p}_{\bar{\d}^0}(\omega;\R^3)$ its two-dimensional representative. In view of measurable selection principles and the lower bound in \eqref{Wgrowth} we find $d\in L^p(\omega;\R^3)$ such that
	\begin{align*}
		\int_\omega W(\Dtwo \bar{u}| d) \dd x = \int_\omega \Wmin(\Dtwo \bar{u}) \dd x.
	\end{align*}
	We extend $d$ trivially to all of $\R^2$ and choose a sequence $(d_k)_k\subset C_c^\infty(\omega;\R^3)$ such that $d_k \to d$ in $L^p(\R^2;\R^3)$ and pointwise almost everywhere. Moreover, we take a sequence $(\bar{u}_k)_k \subset C_c^{\infty}(\R^2;\R^3)$ such that $\bar{u}_k \to \bar{u}$ in $H^{\bar{\rho},p}_{\bar{\d}^0}(\omega;\R^3)$.
	In particular, the continuity of $W$ and its upper bound in \eqref{Wgrowth} produce
	\begin{align}\label{approximation}
		\lim_{k\to \infty} \int_\omega W(\Dtwo \bar{u}_k| d_k) \dd \bar{x} = \int_\omega \Wmin(\Dtwo \bar{u}) \dd \bar{x}
	\end{align}
	due to dominated convergence.
	
	We set $b_{k,\epsilon}:=  \Pcaltwoe d_k$ for every $k\in\N$ with $\Pcaltwoe$ as in \eqref{Ptwo}, and observe that $b_{k,\e}\to b_k:=\Pcaltwo d_k$ as $\epsilon \to 0$ in $W^{1,p}(\R^2;\R^3)$.
    To see this, we first observe 
    \begin{align}\label{Pcal_2d}
    \Pcaltwoe d_k (\bar{x}):=\int_{\R^2}(d_k(\bar{x}+ \bar{\d}^\e\bar{h})-d_k(\bar{x}))\Div V_{\bar\rho}(\bar{h})\dd \bar{h}\quad \text{ for every $\bar{x}\in \omega$,}
    \end{align}
    which is the two-dimensional version of \eqref{Pcal_3d}, and $V_{\bar\rho}$ is the two-dimensional kernel from the nonlocal fundamental theorem of calculus (cf.~\eqref{kernel_fundamental_theorem} and \eqref{eq:Vrhobounds}).
    The Lipschitz continuity of $d_k$ then produces uniform convergence. After taking the derivative in \eqref{Pcal_2d} and switching the integral with the gradient (note that $d_k \in C_c^\infty(\omega;\R^3)$), we repeat the same argument as before to obtain uniform convergence of the derivative as well.
    
    We then define the first step towards a recovery sequence as
	\begin{align*}
		u_{k,\eps}(x):= \mathbbm{1}_{\Omega_{\d^\e}}(x)(\bar{u}_k(\bar{x}) + \eps x_3 b_{k,\e}(\bar{x})) \quad \text{for $x=(\bar{x},x_3) \in \R^3$.}
	\end{align*}
	The rescaled nonlocal gradient of $u_{k,\eps}$ then satisfies
	\begin{align}\label{gradient_recovery}
		(\De u_{k,\eps})T_{(1,\epsilon)}^{-1}&= \big(\Dtwoe \bar{u}_k| 0 \big) + \eps \De(p_3 b_{k,\e})T_{(1,\e)}^{-1}= \big(\Dtwoe \bar{u}_k| 0 \big) + \eps \left(\Qcale (p_3 \bar{\nabla}b_{k,\e})| \frac{1}{\eps}\Qcale b_{k,\e} \right)\nonumber\\
		&= \big(\Dtwoe \bar{u}_k| 0 \big) + \eps \left(\Qcale (p_3 \bar{\nabla}b_{k,\e})| \frac{1}{\eps} \Qcaltwoe \Pcaltwoe d_{k}  \right)\nonumber\\
		&= \big(\Dtwoe \bar{u}_k| d_k \big) + \eps  \big((\Qcale (p_3 \bar{\nabla}b_{k,\e})| 0\big)
	\end{align}
	where $p_3(x) = x_3$ for all $x\in\R^3$ and using that $\Pcaltwoe$ is the inverse of $\Qcaltwoe$. 
    Note also that $\Dtwoe \bar{u}_k \to \Dtwo \bar{u}_k$ in $L^p(\omega;\R^{3\times2})$ due to the smoothness of $\bar{u}_k$, cf.~the strategy in Lemma \ref{le:uniformconvergence}.
    The second term on the right-hand side of \eqref{gradient_recovery} vanishes as $\e\to 0$ since $b_{k,\e}\to b_k$ in $W^{1,p}(\R^2;\R^3)$ and the boundedness of $\Qcal_{\rho,\d^\e}$ as in Lemma \ref{prop:translation} $(i)$.
	The first energy term then satisfies
	\begin{align}\label{first_energy}
		\int_{\Omega}W((\De u_{k,\eps})T_{(1,\e)}^{-1}) \dd x \to \int_\omega W(\Dtwo \bar{u}_k| d_k)\,d\bar{x}\quad \text{ as }\eps \to 0
	\end{align}
	in light of \eqref{gradient_recovery}, the continuity of $W$, \eqref{Wgrowth}, and dominated convergence.
  For the $L^p$-norm, we simply use that $u_{k,\epsilon} \to \bar{u}_k$ as $\epsilon \to 0$ to find
    \begin{align}\label{second_energy}
		\int_{\Omega_{\d^\e}}\abs{u_{k,\epsilon}}^p \dd x \to \int_{\Omega_{\d^0}}\abs{\bar{u}_{k}}^p \dd x\quad \text{ as }\eps \to 0.
\end{align}
Finally, for the double integral we have
\begin{align*}
    \int_{\Omega_{\d^\e}}\int_{\Omega_{\d^\e}}\frac{\abs{u_{k,\epsilon}(\bar{x},x_3)-u_{k,\epsilon}(\bar{x},x_3')}^p}{\abs{\epsilon(x_3-x_3')}}\,dx\,dx' \leq \epsilon^{p-1}\int_{\Omega_{\d^\e}}\int_{\Omega_{\d^\e}}\abs{x_3-x_3'}^{p-1}\abs{b_{k,\e}(\bar{x})}^p\,dx\,dx'\to 0
\end{align*}
as $\epsilon \to 0$ using that $p>1$. Combining \eqref{first_energy} and \eqref{second_energy} with \eqref{approximation} yields
\[
\lim_{k \to \infty} \lim_{\epsilon \to 0} \Ical_\e^{\rm stab}(u_{k,\epsilon})= \int_{\omega}\overline{W}(\Dtwo \bar{u})\,d\bar{x}+\lambda\int_{\Omega_{\d^0}}\abs{\bar{u}}^p \dd x.
\]

\textit{Step 3b: Relaxation.} Let $\bar{u}\in H^{\bar{\rho},p}_{\bar{\d}^0}(\omega;\R^3)$, then we set $v= \Qcal_{\bar{\rho},\bar{\d}^0}\bar{u}\in W^{1,p}(\omega;\R^3)$ so that $\Dtwo \bar{u} = \bar{\nabla}v$ in $\omega$. Standard relaxation results produce a sequence $(v_j)_j\subset W^{1,p}(\omega;\R^3)$ with $v_j\rightharpoonup v$ in $W^{1,p}(\omega;\R^3)$ and
\begin{align}\label{relaxation_step}
    \int_\omega \Wmin(\bar{\nabla}v_j) \dd \bar{x} \to \int_\omega \Wmin^\qc(\bar{\nabla}v) \dd \bar{x}  = \int_\omega \Wmin^\qc(\Dtwo \bar{u}) \dd \bar{x}.
\end{align}
If $E\colon W^{1,p}(\omega;\R^3) \to W^{1,p}(\R^2;\R^3)$ is any bounded linear extension operator, then we define $\tilde{u}_j = \one_{\omega}\Pcaltwo(Ev_j)$ and observe that $\Dtwo\tilde{u}_j = \bar{\nabla}v_j$ so that
\begin{align*}
    \Dtwo\tilde{u}_j \rightharpoonup \Dtwo\bar{u} \quad \text{in $L^p(\omega;\R^{3\times 2})$ as $j \to \infty$}.
\end{align*} 
Combining this with a two-dimensional version of Corollary~\ref{cor:compactness} (which holds in exactly the same way by Remark~\ref{rem:rhobaradmissible}) yields $h_j\in N^{\bar{\rho},p}_{\bar{\d}^0}(\omega;\R^3)$ such that
$\bar{u}_j:=\tilde{u}_j-h_j$ satisfies $\Dtwo\bar{u}_j = \bar{\nabla}v_j$ and $\bar{u}_j\to \bar{u}$ in $L^p(\omega_{\bar{\d}^0};\R^3)$. In particular, \eqref{relaxation_step} and dominated convergence then give
\begin{align*}
    \int_{\omega}\Wmin(\Dtwo \bar{u}_j)\,d\bar{x}+\lambda\int_{\Omega_{\d^0}}\abs{\bar{u}_j}^p \dd x \to \int_{\omega}\Wmin^\qc(\Dtwo \bar{u})\,d\bar{x}+\lambda\int_{\Omega_{\d^0}}\abs{\bar{u}}^p \dd x.
\end{align*}
We then repeat Step 3a for each $j$ and select a suitable diagonal sequence to conclude the proof.
\end{proof}

\begin{remark}
a) The result is phrased with respect to the weak convergence in $L^p$ since we still cannot exclude the functions in $\Ne$. 
Indeed, while the stabilization term ensures that they do not become too large, the fact that $\Ne$ is infinite-dimensional, see~Remark~\ref{rem:N}, makes it so that we still can find bounded sequences in $\Ne$ which do not converge strongly.\smallskip

b) Since the stabilization term $\lambda \Scal_\e$ in \eqref{eq:stabilizedI} is non-negative, the compactness result in Theorem \ref{th:state} ($ii$) still holds. Theorem \ref{th:stabilized} ($ii$) additionally provides the weak convergence of the sequence itself (without taking suitable representatives).\smallskip

c) In light of b), we can also prove the convergence of minimizers as in Corollary \ref{cor:minimizers} when $\Ical_\e$ and $\Ical$ are replaced by their stabilized counterparts. In this setting, in fact, we do not require assumption \eqref{eq:forces}, which was needed in the non-stabilized setting to ensure a generalized force balance and prevent the minimum energy (even for fixed $\d,\e$) from becoming $-\infty$. The stabilized energies avoid this issue by design, so that we merely require the force densities $f_\e,f_0\in L^{p'}(\R^3;\R^3)$ to satisfy
$$f_\e \to f_0 \quad\text{in $L^{p'}(\R^3;\R^3)$ as $\e\to 0$.}$$
\end{remark}

\subsection*{Acknowledgements}
D.E. acknowledges the financial support from the FWF project \href{https://doi.org/10.55776/V1042}{10.55776/V1042} and KU Eichst\"att-Ingolstadt, through which a research visit to TU Wien was made possible. A.M. was funded by the Austrian Science Fund (FWF) project \href{https://doi.org/10.55776/V1042}{10.55776/V1042}. H.S. was funded by the Austrian Science Fund (FWF) projects \href{https://doi.org/10.55776/F65}{10.55776/F65} and \href{https://doi.org/10.55776/Y1292}{10.55776/Y1292} and by the Fonds de la Recherche Scientifique - FNRS through a MIS-Ulysse project (Scientific Impulse Mandate Instrument) number F.6002.25. A visit of H.S. at KU Eichst\"att-Ingolstadt was supported by the Deutsche Forschungsgemeinschaft (DFG) through project 541520348.
For open access purposes, the authors have applied a CC BY public copyright license to any author-accepted manuscript version arising from this submission.

\noindent The authors gratefully acknowledge Anna Dole\v{z}alov\'{a} for her insightful comments and stimulating discussions during the development of this work.
\black

\bibliographystyle{abbrv}
\bibliography{EMS24}

\end{document}